\documentclass[11pt, reqno]{amsart}

\usepackage{amsmath, amsthm, amscd, amsfonts, amssymb, graphicx, color, mathtools, mathrsfs}
\usepackage[bookmarksnumbered, colorlinks, plainpages]{hyperref}
\usepackage{enumerate}
\usepackage{setspace}
\usepackage{multicol}
\usepackage[margin=1in]{geometry}
\usepackage{comment}
\usepackage{booktabs}
\usepackage{caption}
\usepackage{subcaption}
\usepackage{dsfont}
\usepackage{tikz-cd}
\usepackage{algorithmic}
\usepackage{pgfplots}

\theoremstyle{definition}
\newtheorem{theorem}{Theorem}[section]

\newtheorem{proposition}[theorem]{Proposition}
\newtheorem{algorithm}[theorem]{Algorithm}

\numberwithin{equation}{section}

\newcommand{\curl}{\mathrm{curl}\,}

\newcommand{\divg}{\mathrm{div}\,}
\newcommand{\bff}{\boldsymbol}
\newcommand{\bb}{\mathbb}

\newcommand{\dtt}{\mathrm{d}_{\tau}}

\newcommand{\ds}{\mathrm{d}s}

\newcommand{\pa}{\partial}

\newcommand{\norm}[2]{\left\|{#1}\right\|_{#2}}
\newcommand{\inpro}[2]{\left\langle#1,#2\right\rangle}

\newcommand{\abs}[1]{\left|{#1}\right|}

\newcommand{\hdiv}{\mathbb{H}(\mathrm{div})}
\newcommand{\hcurl}{\mathbb{H}(\mathrm{curl})}
\newcommand{\hzerodiv}{\mathbb{H}_0(\mathrm{div})}
\newcommand{\hzerocurl}{\mathbb{H}_0(\mathrm{curl})}
\newcommand{\curlh}{\mathrm{curl}_h\,}

\newcommand{\thetau}{\bff{\theta}_{\bff{u}}}
\newcommand{\rhou}{\bff{\rho}_{\bff{u}}}
\newcommand{\thetap}{\theta_{p}}
\newcommand{\rhop}{\rho_{p}}
\newcommand{\thetab}{\bff{\theta}_{\bff{B}}}
\newcommand{\rhob}{\bff{\rho}_{\bff{B}}}
\newcommand{\thetaj}{\bff{\theta}_{\bff{J}}}
\newcommand{\rhoj}{\bff{\rho}_{\bff{J}}}
\newcommand{\thetae}{\bff{\theta}_{\bff{E}}}
\newcommand{\rhoe}{\bff{\rho}_{\bff{E}}}

\allowdisplaybreaks

\begin{document}
	\setcounter{page}{1}
	
	\title[Error analysis of a div-preserving mixed FEM for the Hall--MHD equations]
	{Error analysis of a divergence-preserving mixed finite element scheme for the incompressible Hall--magnetohydrodynamic equations}
	
	\author[Beniamin Goldys]{Beniamin Goldys}
	\address{School of Mathematics and Statistics, The University of Sydney, Sydney 2006, Australia}
	\email{\textcolor[rgb]{0.00,0.00,0.84}{beniamin.goldys@sydney.edu.au}}
	
	\author[Agus L. Soenjaya]{Agus L. Soenjaya}
	\address{School of Mathematics and Statistics, The University of New South Wales, Sydney 2052, Australia}
	\email{\textcolor[rgb]{0.00,0.00,0.84}{a.soenjaya@unsw.edu.au}}
	
	\author[Thanh Tran]{Thanh Tran}
	\address{School of Mathematics and Statistics, The University of New South Wales, Sydney 2052, Australia}
	\email{\textcolor[rgb]{0.00,0.00,0.84}{thanh.tran@unsw.edu.au}}

		
	\date{May 1, 2026}
	
	\begin{abstract}
	The incompressible Hall-magnetohydrodynamics (Hall--MHD) system presents substantial analytical and computational challenges due to its stiff, highly nonlinear Hall term and the strict requirement that the magnetic field remains solenoidal. In this paper, we study a Voigt-regularised Hall--MHD system, which is of independent analytical interest and provides a physically consistent, well-posed regularisation of the original model.
	We propose, analyse, and implement a structure-preserving, linear, fully discrete finite element method for this regularised problem. Using finite element exterior calculus and a mixed formulation, the spatial discretisation enforces the divergence-free condition on the magnetic field exactly, while a skew-symmetric, linearly implicit time discretisation yields unconditional energy stability. We establish optimal convergence rates for the Voigt-regularised problem and, additionally, derive error estimates for the unregularised Hall--MHD system, with the Voigt regularisation playing a crucial role in the non-resistive regime. Finally, numerical simulations in both 2.5D and 3D corroborate the theoretical results and demonstrate the physical fidelity of the scheme.
	\end{abstract}
	\maketitle

\section{Introduction}
The study of magnetohydrodynamics (MHD) is fundamental to understanding the dynamics of electrically conducting fluids, such as plasmas, liquid metals, and astrophysical flows. Classical MHD models describe the interaction between fluid motion and magnetic fields under a single-fluid approximation, but they fail to capture phenomena at small scales where electron and ion motions decouple. 

The Hall--MHD equations extend classical MHD by incorporating the Hall effect, which becomes significant at length scales comparable to the ion skin depth~\cite{Lig60}. This introduces a dispersive, nonlinear term in the induction equation, leading to richer dynamics, including whistler waves and fast magnetic reconnection corresponding to physically observable phenomena in astrophysics~\cite{Bis96, DraShaSwi08, For91, MorDasGom05}. The Hall term poses significant analytical and numerical challenges due to its strong nonlinearity.

These difficulties motivate the introduction of regularised models that retain the essential physical structure but improve stability and well-posedness.
Among various approaches, the Voigt-type regularisation provides a physically consistent and mathematically convenient modification~\cite{ConPas23, KurLarReb12}. It can be interpreted as a simplified representation of electron inertia and finite Larmor radius effects in the underlying two-fluid plasma description~\cite{HeMao24, HuaHewBroBha25, Hub03}. These mechanisms introduce corrections that smooth the dynamics at spatial scales comparable to the electron inertial length. 

We now introduce the problems considered in this paper.
Let $\mathscr{D}\subset \bb{R}^d$ $(d=2,3)$ be a contractible bounded domain with polytopal boundary. We consider the incompressible resistive Hall--MHD system~\cite{Bis93, Hub03, LaaHuFar23}, which is a coupled system of Navier--Stokes and Maxwell equations, together with a generalised Ohm's law. The problem consists of the velocity field $\bff{u}: (0,T)\times \mathscr{D}\to \bb{R}^3$, the pressure $p: (0,T)\times \mathscr{D}\to \bb{R}$, and the magnetic field $\bff{B}: (0,T)\times \mathscr{D}\to \bb{R}^3$ governed by:
\[
\begin{aligned}
	&\rho \pa_t \bff{u} - \nu \Delta \bff{u} + \rho (\bff{u}\cdot\nabla)\bff{u}
	+ \nabla p - \curl \bff{B}\times \bff{B}
	= \bff{f}
	\quad && \text{in } (0,T)\times \mathscr{D},
	\\[1ex]
	&\pa_t \bff{B}+ \sigma \,\curl (\curl \bff{B}) - \curl (\bff{u}\times \bff{B})
	+ \eta \, \curl (\curl \bff{B}\times \bff{B})
	= \bff{0}
	\quad && \text{in } (0,T)\times \mathscr{D},
	\\[1ex]
	&\divg\bff{u}= \divg\bff{B}= 0
	\quad && \text{in } (0,T)\times \mathscr{D},
	\\[1ex]
	&\bff{u}(0,\bff{x})= \bff{u}_0(\bff{x}), \quad
	\bff{B}(0,\bff{x})= \bff{B}_0(\bff{x})
	\quad && \text{in } \mathscr{D},
	\\[1ex]
	&\bff{u}=\bff{0}, \quad \bff{B}\cdot \bff{n}=0, \quad
	\curl \bff{B}\times \bff{n}=\bff{0}, \quad
	(\curl \bff{B}\times \bff{B})\times \bff{n}=\bff{0}
	\quad && \text{on } (0,T)\times \partial\mathscr{D},
\end{aligned}
\]
where $\bff{n}$ is the outward pointing normal vector on $\pa \mathscr{D}$, and the initial data satisfy~$\divg \bff{u}_0=\divg \bff{B}_0=0$. This system is the form of the Hall--MHD equations in bounded domains most commonly studied in the literature.

The positive coefficients $\rho$, $\nu$, $\sigma$, and $\eta$ represent the fluid density, fluid kinematic viscosity, magnetic resistivity, and the strength of the Hall effect, respectively. For clarity of exposition, subsequently we set $\rho=1$, $\bff{f}=\bff{0}$, and assume $d=3$. The case $d=2$ corresponds to the 2.5D Hall--MHD, where the system is defined on a two-dimensional domain but retains all three vector components under the assumption of vanishing partial derivatives of the unknowns in the $z$-component~\cite{DonSer12, LaaHuFar23, MajBer02}. We provide a brief discussion of this reduction in Section~\ref{sec:25d}.

To design a structure-preserving numerical method, we will rewrite the problem including the underlying physical quantities, namely the electric field $\bff{E}: (0,T)\times \mathscr{D}\to \bb{R}^3$ and the current density $\bff{J}: (0,T)\times \mathscr{D}\to \bb{R}^3$. The Hall--MHD system is 
then rewritten in the following mixed formulation:
\begin{subequations}\label{equ:mhd}
	\begin{alignat}{2}
		&\pa_t \bff{u} - \nu \Delta \bff{u} + (\bff{u}\cdot\nabla)\bff{u}  + \nabla p - \bff{J}\times \bff{B}
		= 
		\bff{0}
		\; && \quad\text{in } (0,T)\times \mathscr{D},
		\label{equ:mhd ns}
		\\[1ex]
		&\pa_t \bff{B}+ \curl \bff{E} = \bff{0}
		\; && \quad\text{in } (0,T)\times \mathscr{D},
		\label{equ:mhd induction}
		\\[1ex]
		&\sigma\bff{J}+\eta \bff{J}\times \bff{B} = \bff{E}+\bff{u}\times \bff{B}
		\; && \quad\text{in } (0,T)\times \mathscr{D},
		\label{equ:mhd ohm}
		\\[1ex]
		&\bff{J}- \curl \bff{B} = \bff{0}
		\; && \quad\text{in } (0,T)\times \mathscr{D},
		\label{equ:mhd ampere}
		\\[1ex]
		&\divg\bff{u}= \divg\bff{B}= 0
		\; && \quad\text{in } (0,T)\times \mathscr{D},
		\label{equ:mhd div}
		\\[1ex]
		&\bff{u}(0,\bff{x})= \bff{u}_0(\bff{x}), \quad
		\bff{B}(0,\bff{x})= \bff{B}_0(\bff{x})
		\; && \quad\text{in } \mathscr{D},
		\label{equ:mhd init}
		\\[1ex]
		&\bff{u}=\bff{0}, \quad \bff{B}\cdot \bff{n}=0, \quad \bff{E} \times \bff{n}=\bff{0}, \quad \bff{J}\times \bff{n}=\bff{0}
		\; && \quad\text{on } (0,T)\times \partial \mathscr{D}.
		\label{equ:mhd boundary}
	\end{alignat}
\end{subequations}
In this formulation, \eqref{equ:mhd ns} denotes the momentum equation, while \eqref{equ:mhd induction} and \eqref{equ:mhd ohm} represent to the induction equation and generalised Ohm's law, respectively. Equation~\eqref{equ:mhd ampere} is the Amp\`ere law. The divergence-free constraints in \eqref{equ:mhd div} express two physical principles: incompressibility of the fluid yields $\divg \bff{u}=0$, whereas $\divg \bff{B}=0$ reflects the absence of magnetic monopoles (Gauss' law). The boundary conditions in \eqref{equ:mhd boundary} correspond to the perfectly conducting boundaries considered in \cite{LaaHuFar23}.

As discussed previously, the Hall--MHD system exhibits strong nonlinear and dispersive behaviour due to the term~$\eta \bff{J}\times \bff{B}$ in the generalised Ohm's law~\eqref{equ:mhd ohm}, leading to steep gradients and high-frequency oscillations caused by the whistler waves that make numerical simulation and analysis difficult~\cite{ArnDreGra08, GuoMeiYan24, Hub03, Iwa25}. In particular, resolving the small scales induced by the Hall term often requires extremely fine meshes or small time steps, and the global existence of strong solutions in three dimensions remains an open problem. 

These difficulties motivate the introduction of Voigt-type regularisation that retains the essential physical structure, but improve stability and well-posedness~\cite{ConPas23, KurLarReb12, LarTit14}.
In the Voigt model, this is achieved by replacing the time derivative with its spatially filtered version $(I-\alpha \Delta) \partial_t$, where $\alpha>0$ denotes the regularisation length scale and $I$ is the identity operator. This modification preserves the divergence-free and energy structures of the original Hall--MHD system, while improving numerical stability and permitting coarser spatial discretisations. Beyond its numerical advantages, the Voigt regularisation is also of independent analytical interest, as it yields globally well-posed models with enhanced regularity while remaining asymptotically consistent with the original system as $\alpha\to 0$; see, e.g.,~\cite{BrzLarSaf25, ConPas23, LarTit14}.

Motivated by these considerations, we consider a Voigt-regularised Hall--MHD system~\cite{HuaHewBroBha25}, obtained by applying a spatial filter of the form $(I-\alpha_1 \Delta) \partial_t$ to the time derivative in the momentum equation \eqref{equ:mhd ns} and retaining the electron inertia term $\alpha_2 \partial_t \bff{J}$ on the right-hand side of \eqref{equ:mhd ohm}. Despite this regularisation, the equations remain strongly nonlinear due to the convective, Lorentz, and Hall coupling terms, which continue to govern the complex interaction between velocity and magnetic fields. Under this reformulation, the Voigt-regularised Hall--MHD system takes the form:
\begin{subequations}\label{equ:reg mhd}
	\begin{alignat}{2}
		&(I-\alpha_1\Delta)\pa_t \bff{u} - \nu \Delta \bff{u} + (\bff{u}\cdot\nabla)\bff{u}  + \nabla p - \bff{J}\times \bff{B}
		= 
		\bff{0}
		\; && \quad\text{in } (0,T)\times \mathscr{D},
		\label{equ:reg mhd ns}
		\\[1ex]
		&\pa_t \bff{B}+  \curl \bff{E} = \bff{0}
		\; && \quad\text{in } (0,T)\times \mathscr{D},
		\label{equ:reg mhd induction}
		\\[1ex]
		&\alpha_2 \pa_t \bff{J} + \sigma\bff{J} + \eta \bff{J}\times \bff{B} =\bff{E} + \bff{u}\times \bff{B}
		\; && \quad\text{in } (0,T)\times \mathscr{D},
		\label{equ:reg mhd electric}
		\\[1ex]
		&\bff{J}- \curl \bff{B} = \bff{0}
		\; && \quad\text{in } (0,T)\times \mathscr{D},
		\label{equ:reg mhd ampere}
		\\[1ex]
		&\divg\bff{u}= \divg\bff{B}= 0
		\; && \quad\text{in } (0,T)\times \mathscr{D},
		\label{equ:reg mhd div}
		\\[1ex]
		&\bff{u}(0,\bff{x})= \bff{u}_0(\bff{x}), \quad
		\bff{B}(0,\bff{x})= \bff{B}_0(\bff{x})
		\; && \quad\text{in }  \mathscr{D},
		\label{equ:reg mhd init}
		\\[1ex]
		&\bff{u}=\bff{0}, \quad \bff{B}\cdot \bff{n}=0, \quad \bff{E}\times \bff{n}=\bff{0}, \quad \bff{J} \times \bff{n}=\bff{0}
		\; && \quad\text{on } (0,T)\times \partial \mathscr{D},
		\label{equ:reg mhd boundary}
	\end{alignat}
\end{subequations}
For simplicity, periodic boundary conditions are frequently employed in practice, and our analysis remains valid in this setting. We remark that Voigt-type regularisation has been investigated in related contexts, including the Euler, Navier--Stokes, and standard MHD equations~\cite{BrzLarSaf25, KurLarReb12, LarTit14, YanHuaHe24}.

Various numerical algorithms for solving the Hall--MHD problem have recently been proposed. For the \emph{stationary} resistive Hall--MHD system, a structure-preserving finite element method (FEM) and an efficient preconditioner are developed in~\cite{LaaHuFar23}. For the time-dependent problem, several works have employed scalar auxiliary variable (SAV) approaches. Specifically, \cite{GuoMeiYan24} proposes a Legendre--Galerkin spectral method for bounded domains, while \cite{GuoYanWei25} develops a Hermite--Galerkin spectral method to handle variable density flows on unbounded domains. Concurrently, \cite{WeiMei26} combines a BDF2-type integrator with a nonlocal SAV-FEM approach, and a second-order pressure projection for the fluid. From a purely solver-oriented perspective, \cite{Cha25} utilises a multigrid Newton--Krylov method to efficiently advance the implicit time stepping. 

However, the aforementioned SAV-based methods share notable mathematical limitations: they dissipate a modified energy functional rather than the true physical energy, and they may fail to enforce the divergence-free condition exactly at the discrete level. Furthermore, the spectral methods in~\cite{GuoMeiYan24, GuoYanWei25} are practically restricted to canonical geometries (e.g., idealised boxes or infinite spaces), lacking the geometric flexibility required to model complex physical boundaries. Most importantly, rigorous error analysis remains unavailable for all of the above schemes.

Preserving the divergence-free condition at the discrete level is essential for both stability and physical fidelity, as violations can induce nonphysical magnetic monopoles and trigger severe numerical instabilities~\cite{BraBar80, DaiWoo98}. To address these challenges, we develop and rigorously analyse a structure-preserving finite element scheme for the Voigt-regularised Hall--MHD system~\eqref{equ:reg mhd}, which serves as a mathematically well-posed and physically consistent approximation of the original equations~\eqref{equ:mhd}. The proposed method enforces the solenoidal constraint exactly through finite element exterior calculus, and satisfies a discrete energy law by means of a skew-symmetric, semi-implicit time discretisation. We establish optimal convergence rates for the Voigt-regularised problem under suitable regularity assumptions. In addition, the analysis extends to the unregularised Hall--MHD system, for which convergence can still be shown, although the Voigt regularisation plays a crucial role in the non-resistive regime ($\sigma=0$ in \eqref{equ:reg mhd electric}).

The remainder of this paper is organised as follows. Section 2 introduces the functional setting, notation, and necessary preliminary results. In Section 3, we establish the stability of the proposed scheme and carry out a rigorous error analysis, culminating in the proof of the main convergence theorem (Theorem~\ref{the:main}). Finally, Section~\ref{sec:num sim} presents a series of physically relevant numerical experiments, in both 2.5D and fully 3D cases, to validate the theoretical convergence rates and demonstrate the robustness of the structure-preserving framework.

\section{Preliminaries}

\subsection{Notations}
We begin by defining some notations used in this paper. Let $\mathscr{D}\subset \bb{R}^d$, $d\in\{2,3\}$, be a bounded domain with polytopal boundary. For $p\in [1,\infty]$, the function space $\bb{L}^p := \bb{L}^p(\mathscr{D}; \bb{R}^d)$ denotes the space of $p$-th integrable functions on~$\mathscr{D}$ taking values in $\bb{R}^d$, and $\bb{W}^{s,p} := \bb{W}^{s,p}(\mathscr{D}; \bb{R}^d)$ denotes the Sobolev space of functions on $\mathscr{D}$ taking values in $\bb{R}^d$. We
write $\bb{H}^s := \bb{W}^{s,2}$ and set $\bb{W}^{0,p}:=\bb{L}^p$. We denote by $L^2_0$ the subspace of scalar-valued functions in $L^2(\mathscr{D})$ with zero average. The space $\mathcal{C}^0(\mathscr{D};\bb{R}^d)$ denotes the space of continuous functions on $\mathscr{D}$ taking values in $\bb{R}^d$, and we write $C^0(\mathscr{D}):= \mathcal{C}^0(\mathscr{D};\bb{R})$.

We shall introduce the following function spaces:
\begin{align*}
	\bb{H}(\mathrm{div}) &:= \{\bff{v}\in \bb{L}^2: \divg \bff{v}\in L^2\},
	\\
	\bb{H}(\mathrm{curl}) &:= \{\bff{v}\in \bb{L}^2: \curl \bff{v}\in \bb{L}^2\}.
\end{align*}
We further introduce the spaces $\bb{H}^1_0$, $\hzerodiv$, and $\hzerocurl$ as the subspaces of functions in $\bb{H}^1$, $\hdiv$, and $\hcurl$ with zero standard, normal, and tangential traces in $\bb{H}^{\frac12}(\partial\mathscr{D})$, $\bb{H}^{-\frac12}(\partial\mathscr{D})$, and $\bb{H}^{-\frac12}(\partial\mathscr{D})$, respectively.

If $X$ is a Banach space, $L^p(0,T; X)$ and $W^{k,p}(0,T;X)$ denote respectively the usual Lebesgue and Sobolev spaces of strongly measurable functions on $(0,T)$ taking values in $X$. The space $C^0([0,T];X)$ denotes the space of continuous functions on $[0,T]$ taking values in $X$. For brevity, we will denote the spaces $L^p(0,T;X)$, $W^{k,p}(0,T;X)$, and $C^0([0,T];X)$ by $L^p_T(X)$, $W^{k,p}_T(X)$, and $C^0_T(X)$, respectively. 

Throughout this paper, we denote the scalar product in a Hilbert space $H$ by $\inpro{\cdot}{\cdot}_H$ and its corresponding norm by $\|\cdot\|_H$. We will not distinguish between the scalar product of $\bb{L}^2$ vector-valued functions taking values in $\bb{R}^3$ and the scalar product of $\bb{L}^2$ matrix-valued functions taking values in $\bb{R}^{3\times 3}$, and denote them by $\langle\cdot,\cdot\rangle$.

Finally, the constant $C$ in the estimate denotes a generic constant, which may take different values in different occurrences. If its dependence on a particular variable, e.g., $S$, needs to be emphasised, we write $C_S$.

\subsection{Finite element spaces}

We now introduce the finite element spaces employed in the numerical approximation of problem~\eqref{equ:reg mhd}. Let $\{\mathcal{T}_h\}_{h>0}$ be a family of quasi-uniform triangulations of $\mathscr{D}\subset \bb{R}^d$ into simplicial elements with maximal mesh-size $h$. Let $\bb{P}_k$ and $P_k$ denote, respectively, the space of vector-valued and scalar-valued polynomials of degree at most $k$. Let $(\bb{V}_h, Q_h)$ denote the MINI finite element pair, i.e. the lowest-order inf-sup stable conforming pair for incompressible flow~\cite{Vol16}. More precisely, for any $K\in\mathcal{T}_h$, let $\{\lambda_i^K\}_{i=1}^{d+1}$ be the barycentric coordinates associated with $K$. We construct the bubble function on $K$ by $\bff{b}_K:= \prod_{i=1}^{d+1} \lambda_i^K\in \bb{P}_{d+1}(K)$, and define
\begin{align}\label{equ:Vh vec}
	&\bb{V}_{h}:= \{\bff{v}_h \in \mathcal{C}^0(\overline{\mathscr{D}}; \bb{R}^d): \bff{v}_h|_K \in \bb{P}_1(K) \oplus \mathrm{span}\{\bff{b}_K\}, \; \forall K \in \mathcal{T}_h\} \subset \bb{H}^1,
	\\
	\label{equ:Qh}
	&Q_h := \{q_h\in C^0(\overline{\mathscr{D}}): q_h|_K\in P_1(K), \; \forall K \in \mathcal{T}_h\} \subset H^1.
\end{align}
We denote by $\bb{X}_h$, $\bb{RT}_h$, and $DG_h$ the lowest-order N\'ed\'elec space, the lowest-order Raviart--Thomas finite element space, and the space of scalar-valued piecewise constant functions, respectively. More precisely,
\begin{align}
	\label{equ:Xh}
	\bb{X}_h &:= \{\bff{v}_h \in \bb{L}^2: \bff{v}_h|_K \in \bb{P}_0(K) \oplus \bff{x}\times \bb{P}_0(K),\; \forall K\in \mathcal{T}_h\} \subset \hcurl,
	\\
	\label{equ:RTh}
	\bb{RT}_h &:= \{\bff{v}_h \in \bb{L}^2: \bff{v}_h|_K \in \bb{P}_0(K) \oplus \bff{x} P_0(K),\; \forall K\in \mathcal{T}_h\} \subset \hdiv,
	\\
	\label{equ:DGh}
	DG_h &:= \{q_h\in L^2(\mathscr{D}): q_h|_K\in P_0(K), \; \forall K \in \mathcal{T}_h\} \subset L^2.
\end{align}
The corresponding finite element spaces with zero (standard, tangential, or normal) traces are $\bb{V}_h^0:= \bb{V}_h\cap \bb{H}^1_0$, $\bb{X}_h^0:= \bb{X}_h\cap \hzerocurl$, and $\bb{RT}_h^0 := \bb{RT}_h \cap \hzerodiv$. Moreover, let $Q_h^0:= Q_h\cap L^2_0$ and $V_h^0:= Q_h\cap H^1_0$.

For contractible domains in $\bb{R}^3$, we have the following de Rham exact sequence structure for the function spaces:
\begin{equation}\label{equ:exact cont}
	\begin{tikzcd}
		H^1_0 \arrow[r, "\nabla"] & \mathbb{H}_0(\mathrm{curl}) \arrow[r, "\mathrm{curl}"] & \mathbb{H}_0(\mathrm{div}) \arrow[r, "\mathrm{div}"] & L^2_0.
	\end{tikzcd}
\end{equation}
and for the finite element spaces:
\begin{equation}\label{equ:exact fe}
	\begin{tikzcd}
		Q_h^0 \arrow[r, "\nabla"] & \mathbb{X}_h^0 \arrow[r, "\mathrm{curl}"] & \mathbb{RT}_h^0 \arrow[r, "\mathrm{div}"] & DG_h.
	\end{tikzcd}
\end{equation}
The above de Rham complexes are linked by the canonical interpolators, or the corresponding quasi-interpolators  $\mathcal{I}_h$ defined in~\cite{ErnGue17}, leading to the following commutative diagram~\cite{ArnFalWin06}:
\begin{equation}\label{equ:comm}
	\begin{tikzcd}
		H^1_0 \arrow[r, "\nabla"] \arrow[d, "\mathcal{I}_h^Q"] & \hzerocurl \arrow[r, "\mathrm{curl}"] \arrow[d, "\mathcal{I}_h^{\bb{X}}"] & \hzerodiv \arrow[r, "\mathrm{div}"] \arrow[d, "\mathcal{I}_h^{\bb{RT}}"] & L^2_0 \arrow[d, "\mathcal{I}_h^{DG}"] \\
		Q_{h}^0 \arrow[r, "\nabla"] & \mathbb{X}_h^0 \arrow[r, "\mathrm{curl}"]                                                     & \mathbb{RT}_h^0 \arrow[r, "\mathrm{div}"]                     & DG_h              
	\end{tikzcd}
\end{equation}

\subsection{Analytical tools}\label{sec:analytic}

Let $\bb{FE}_h$ be one of the finite element spaces $\bb{V}_h^0$, $\bb{X}_h^0$, or $\bb{RT}_h^0$. We denote by $\Pi_h^{\bb{FE}}$ the orthogonal projection onto $\bb{FE}_h$ with respect to the $\bb{L}^2$-inner product, satisfying
\begin{align}\label{equ:Pi h L2}
	\inpro{\Pi_h^{\bb{FE}} \bff{v}-\bff{v}}{\bff{\phi}_h}=0, \quad \forall \bff{\phi}_h\in \bb{FE}_h.
\end{align}
This orthogonality relation implies that $\Pi_h^{\bb{FE}} \bff{v}$ is the best approximation of $\bff{v}$ in $\bb{FE}_h$ with respect to the $\bb{L}^2$ norm. Consequently, by~\cite[Theorem~2.4]{ErnGue18}, for any $s\in [0,1]$, there exists a constant $C$ depending on $s$ such that
\begin{align}\label{equ:PiV approx}
	\norm{\bff{v}-\Pi_h^{\bb{V}} \bff{v}}{\bb{L}^2}
	&\leq
	Ch^{1+s} \norm{\bff{v}}{\bb{H}^{1+s}},
	\\
	\label{equ:PiX approx}
	\norm{\bff{v}-\Pi_h^{\mathbb{X}} \bff{v}}{\bb{L}^2}
	&\leq
	Ch^{s} \norm{\bff{v}}{\bb{H}^{s}},
	\\
	\label{equ:PiRT approx}
	\norm{\bff{v}-\Pi_h^{\mathbb{RT}} \bff{v}}{\bb{L}^2}
	&\leq
	Ch^s \norm{\bff{v}}{\bb{H}^{s}}.
\end{align}


The quasi-interpolation operators $\mathcal{I}_h^{\bb{X}}$ and $\mathcal{I}_h^{\bb{RT}}$ in~\eqref{equ:comm} satisfy the following approximation properties~\cite{ErnGue18}: If $s\in(0,1]$ and $p\in[1,\infty)$, then there exists a constant $C$ depending on $s$ and $p$ such that
\begin{align}\label{equ:interp ned approx}
	\norm{\bff{v}- \mathcal{I}_h^{\bb{X}} \bff{v}}{\bb{L}^p}
	&\leq
	Ch^s \norm{\bff{v}}{\bb{W}^{s,p}},
	\\
	\label{equ:interp RT approx}
	\norm{\bff{v}- \mathcal{I}_h^{\bb{RT}} \bff{v}}{\bb{L}^p}
	&\leq
	Ch^s \norm{\bff{v}}{\bb{W}^{s,p}}.
\end{align}
Similarly, we also have the following stability property~\cite{BeiHuMas25} for the projector $\Pi_h^{\mathbb{X}}$:
\begin{align}\label{equ:proj Pi stab}
	\norm{\Pi_h^{\bb{X}} \bff{v}}{\bb{L}^p}
	&\leq
	C \norm{\bff{v}}{\bb{L}^p}, \quad p\in [1,\infty].
\end{align}

The MINI finite element pair $(\bb{V}_{h}^0, Q_h^0)$ satisfies the discrete inf-sup condition~\cite{Vol16}: There exists a positive constant $\beta$, depending only on $\mathscr{D}$, such that
\begin{align}\label{equ:inf sup}
	\inf_{0\neq q_h\in Q_h^0} \sup_{\bff{0}\neq \bff{v}_h\in \bb{V}_{h}^0} \frac{\inpro{\divg \bff{v}_h}{q_h}}{\norm{\nabla\bff{v}_h}{\bb{L}^2} \norm{q_h}{L^2}} \geq \beta.
\end{align}

We shall need the Stokes projection operator for the fluid variables. More precisely, we define the Stokes projector $\mathcal{S}_h: \bb{H}^1_0\times L^2_0\to \bb{V}_{h}^0\times Q_h^0$ by $\mathcal{S}_h(\bff{v},q):=(\mathcal{S}_h \bff{v}, \mathcal{S}_h q)$ such that 
\begin{subequations}\label{equ:stokes proj}
	\begin{alignat}{2}
		&\nu \inpro{\nabla \mathcal{S}_h \bff{v}-\nabla \bff{v}}{\nabla\bff{\phi}_h}- \inpro{\mathcal{S}_h q - q}{\divg \bff{\phi}_h}= 0,
		\; && \quad\forall \bff{\phi}_h \in \bb{V}_{h}^0,
		\label{equ:stokes proj u}
		\\[1ex]
		&\inpro{\divg \mathcal{S}_h \bff{v}- \divg \bff{v}}{q_h} = 0,
		\; && \quad\forall q_h \in Q_h.
		\label{equ:stokes proj p}
	\end{alignat}
\end{subequations}
The Stokes projector satisfies the following approximation property~\cite{GaoQiu19, Vol16}: For $s\in [0,1]$,
\begin{align}\label{equ:Stokes approx}
	\norm{\bff{v}- \mathcal{S}_h \bff{v}}{\bb{H}^1} + \norm{q-\mathcal{S}_h q}{L^2}
	&\leq
	Ch^s \left(\norm{\bff{v}}{\bb{H}^{1+s}}+ \norm{q}{H^s}\right).
\end{align}

As in \cite{HuXu19, LaaHuFar23}, we utilise the discrete curl operator $\curlh: \bb{L}^2 \to \bb{X}_h^0$ defined by
\begin{align}\label{equ:curlh}
	\inpro{\curlh \bff{C}_h}{\bff{D}_h}= \inpro{\bff{C}_h}{\curl \bff{D}_h}, \quad \forall \bff{D}_h\in \bb{X}_h^0.
\end{align}
For any $\bff{C}_h\in \bb{RT}_h^0$, the generalised discrete Gaffney inequality~\cite{HeHuXu19} holds:
\begin{align}\label{equ:gaffney}
	\norm{\bff{C}_h}{\bb{L}^{3+\delta}} \leq \norm{\curlh \bff{C}_h}{\bb{L}^2} + \norm{\divg \bff{C}_h}{\bb{L}^2},
\end{align}
where $\delta\in [0,3]$ depends on the regularity of $\mathscr{D}$.

\subsection{Problem formulation and numerical scheme}

A weak formulation of the problem \eqref{equ:reg mhd} can be written as follows: Let the initial data $(\bff{u}(0),\bff{B}(0))= (\bff{u}_0,\bff{B}_0)$ satisfy the compatibility conditions $\bff{u}_0\in \bb{H}^1_0$ with $\divg \bff{u}_0=0$, and $\bff{B}_0\in \hzerodiv\cap \hcurl\cap \bb{L}^\infty$ with $\divg \bff{B}_0=0$ and $\bff{J}(0)=\curl \bff{B}_0$. We seek $(\bff{u}, p, \bff{B}, \bff{E}, \bff{J})$ with 
\begin{equation}\label{equ:exa sol req}
	\left\{
	\begin{aligned}
		&\bff{u}\in H^1_T(\bb{H}^1_0), \quad 
		p\in L^2_T(L^2_0), \quad
		\bff{J} \in L^\infty_T(\bb{L}^2)\cap H^1_T(\hzerocurl'),
		\\
		&\bff{B}\in H^1_T(\hzerodiv') \cap L^\infty_T(\hzerodiv) \cap L^2_T(\bb{L}^\infty), \quad
		\bff{E} \in L^2_T(\hzerocurl),
	\end{aligned}
	\right.
\end{equation}
such that
\begin{subequations}\label{equ:weak reg mhd}
	\begin{alignat}{2}
		&\inpro{\pa_t \bff{u}}{\bff{\phi}} + \alpha_1 \inpro{\nabla \pa_t \bff{u}}{\nabla \bff{\phi}} 
		+ \nu \inpro{\nabla \bff{u}}{\nabla \bff{\phi}} + \inpro{(\bff{u}\cdot\nabla)\bff{u}}{\bff{\phi}} 
		\; && 
		\nonumber\\
		&\qquad\qquad - \inpro{p}{\divg \bff{\phi}} - \inpro{\bff{J}\times \bff{B}}{\bff{\phi}}
		= 
		0,
		\; && \quad \forall \bff{\phi}\in \bb{H}^1_0,
		\label{equ:weak reg mhd u}
		\\[1ex]
		&\inpro{\pa_t \bff{B}}{\bff{\psi}} + \inpro{\curl \bff{E}}{\bff{\psi}}= 0,
		\; && \quad \forall \bff{\psi} \in \hzerodiv,
		\label{equ:weak reg mhd B}
		\\[1ex]
		&\alpha_2 \inpro{\pa_t \bff{J}}{\bff{\chi}}+ \sigma \inpro{\bff{J}}{\bff{\chi}} + \eta \inpro{\bff{J}\times \bff{B}}{\bff{\chi}}
		\; && 
		\nonumber\\
		&\qquad\qquad
		= 
		\inpro{\bff{E}}{\bff{\chi}} + \inpro{\bff{u}\times \bff{B}}{\bff{\chi}},
		\; && \quad \forall \bff{\chi} \in \hzerocurl,
		\label{equ:weak reg mhd E}
		\\[1ex]
		&\inpro{\bff{J}}{\bff{\omega}}- \inpro{\bff{B}}{\curl\bff{\omega}} = 0,
		\; && \quad\forall \bff{\omega} \in \hzerocurl,
		\label{equ:weak reg mhd J}
		\\[1ex]
		&\inpro{\divg\bff{u}}{q} = 0,
		\; && \quad\forall q\in L^2_0.
		\label{equ:weak reg mhd c}
	\end{alignat}
\end{subequations}
Here, $\hzerodiv'$ and $\hzerocurl'$ denote, respectively, the dual of $\hzerodiv$ and $\hzerocurl$ with respect to the $L^2$ inner product.

Let $\tau>0$ denote the uniform time step and $t_n=n\tau$ for $n=0,1,\ldots,N$, where $N=\lfloor{T/\tau}\rfloor$. For any time-discrete function $\bff{v}$, we write $\bff{v}^n:= \bff{v}(t_n)$ and
\begin{align*}
	\dtt \bff{v}^n:= \frac{\bff{v}^n-\bff{v}^{n-1}}{\tau}.
\end{align*}

A fully discrete divergence-preserving and linear mixed finite element scheme for solving problem~\eqref{equ:weak reg mhd} can be described as follows.

\begin{algorithm}[Linear fully discrete mixed FEM scheme for regularised Hall--MHD]
	\label{alg:linear fem mhd}
	\begin{algorithmic}[1]
		Let
		\[
		\bb{Y}_h := \bb{V}_h^0 \times Q_h^0 \times \bb{RT}_h^0 \times \bb{X}_h^0 \times \bb{X}_h^0.
		\]
		be the finite element space.
		\STATE \textbf{Given:} mesh size $h$,
		time step $\tau$, initial data 
		$(\bff{u}_h^0, \bff{B}_h^0)= (\mathcal{S}_h \bff{u}_0, \mathcal{I}_h^{\bb{RT}} \bff{B}_0)$, and $\bff{J}_h^0=\curlh \bff{B}_h^0$.
		
		\FOR{$n = 1,2,\dots,N$}
		\STATE Given $(\bff{u}_h^{n-1}, p_h^{n-1}, \bff{B}_h^{n-1}, \bff{E}_h^{n-1}, \bff{J}_h^{n-1})$,
		solve the following \emph{linear system}:
		
		\STATE Find $(\bff{u}_h^n, p_h^n, \bff{B}_h^n, \bff{E}_h^n, \bff{J}_h^n) \in \bb{Y}_h$ such that
		for all $(\bff{\phi}_h, q_h, \bff{\psi}_h, \bff{\chi}_h, \bff{\omega}_h) \in \bb{Y}_h$,
		\begin{subequations}\label{equ:fem reg mhd}
			\begin{align}
				&\inpro{\dtt \bff{u}_h^n}{\bff{\phi}_h}
				+ \alpha_1 \inpro{\nabla \dtt \bff{u}_h^n}{\nabla \bff{\phi}_h}
				+ \nu \inpro{\nabla \bff{u}_h^n}{\nabla \bff{\phi}_h}
				\nonumber\\
				&\quad
				+ \frac12 \Big[
				\inpro{(\bff{u}_h^{n-1}\cdot \nabla)\bff{u}_h^n}{\bff{\phi}_h}
				- \inpro{(\bff{u}_h^{n-1}\cdot \nabla)\bff{\phi}_h}{\bff{u}_h^n}
				\Big]
				\nonumber\\
				&\quad
				- \inpro{p_h^n}{\divg \bff{\phi}_h}
				- \inpro{\bff{J}_h^n \times \bff{B}_h^{n-1}}{\bff{\phi}_h)}
				= 0,
				\label{equ:fem reg mhd a}
				\\[1ex]
				&\inpro{\dtt \bff{B}_h^n}{\bff{\psi}_h}
				+ \inpro{\curl \bff{E}_h^n}{\bff{\psi}_h} = 0,
				\label{equ:fem reg mhd b}
				\\[1ex]
				&\alpha_2 \inpro{\dtt \bff{J}_h^n}{\bff{\chi}_h}
				+ \sigma \inpro{\bff{J}_h^n}{\bff{\chi}_h}
				+ \eta \inpro{\bff{J}_h^n \times \bff{B}_h^{n-1}}{\bff{\chi}_h}
				\nonumber\\
				&\quad
				= \inpro{\bff{E}_h^n}{\bff{\chi}_h}
				+ \inpro{\bff{u}_h^n \times \bff{B}_h^{n-1}}{\bff{\chi}_h)},
				\label{equ:fem reg mhd c}
				\\[1ex]
				&\inpro{\bff{J}_h^n}{\bff{\omega}_h}
				- \inpro{\bff{B}_h^n}{\curl \bff{\omega}_h} = 0,
				\label{equ:fem reg mhd d}
				\\[1ex]
				&\inpro{\divg \bff{u}_h^n}{q_h} = 0.
				\label{equ:fem reg mhd e}
			\end{align}
		\end{subequations}
		
		\ENDFOR
	\end{algorithmic}
\end{algorithm}

In Algorithm~\ref{alg:linear fem mhd}, other choices for the initial data approximation, for instance $(\mathcal{I}_h^{\bb{V}}\bff{u}_0, \mathcal{I}_h^{\bb{RT}} \bff{B}_0)$, are possible, provided they approximate the exact initial conditions accurately, and more importantly preserves the divergence-free condition on $\bff{B}_0$. 
Well-posedness, stability, convergence, and structure-preserving properties of this mixed scheme will be established in Section~\ref{sec:conv}.

In our error analysis, we assume that \eqref{equ:weak reg mhd} admits a sufficiently regular solution, namely, in addition to~\eqref{equ:exa sol req} we have, for some $s>0$,
\begin{equation}\label{equ:assum}
	\bff{u},\bff{B},\bff{J} \in W^{1,\infty}_T(\bb{H}^{1+s}) \cap H^2_T(\bb{L}^2),
	\quad
	\text{and}
	\quad
	\bff{E} \in L^\infty_T(\bb{H}^{1+s}).
\end{equation}
Under these regularity assumptions, our main convergence result is stated in the following theorem, the detailed proof of which will be presented in Section~\ref{sec:conv}. In particular, for $\sigma>0$, the error estimate holds even for $\alpha_1=\alpha_2=0$, corresponding to the unregularised Hall--MHD system.
\begin{theorem}\label{the:main}
	Let $(\bff{u}_h^n, p_h^n, \bff{B}_h^n, \bff{E}_h^n, \bff{J}_h^n)$ be computed by Algorithm~\ref{alg:linear fem mhd}, and let $(\bff{u}, p ,\bff{B},\bff{E}, \bff{J})$ be the solution of \eqref{equ:weak reg mhd} with regularity depicted in~\eqref{equ:assum}. Then, for $n\in \{1,2,\ldots,N\}$,
	\begin{align*}
		\norm{\bff{u}_h^n-\bff{u}^n}{\bb{L}^2}
		+
		\alpha_1 \norm{\nabla \bff{u}_h^n- \nabla \bff{u}^n}{\bb{L}^2}
		+
		\norm{\bff{B}_h^n-\bff{B}^n}{\bb{L}^2}
		+
		\alpha_2 \norm{\bff{J}_h^n-\bff{J}^n}{\bb{L}^2}
		\leq
		C(h^{s}+\tau).
	\end{align*}
	If $\sigma>0$, then the constant $C$ depends on $T$ and possibly on $\alpha_1,\alpha_2$, but is independent of $n$, $h$, $\tau$, and remains bounded as $\alpha_1,\alpha_2\to 0$.
	If $\sigma=0$, this estimate remains valid, with $C$ depending on $T$ and $\alpha_2^{-1}$.
\end{theorem}

\subsection{The 2.5D Hall--MHD ($d=2$)}\label{sec:25d}

Let $\mathscr{D}\subset \bb{R}^2$.
The 2.5D Hall--MHD, which assumes translational invariance along the $z$-axis while retaining all three components of the velocity and magnetic vector fields~\cite{DonSer12, LaaHuFar23, MajBer02}, will be discussed briefly. Under this assumption, the 3D unknowns decouple into in-plane vectors and out-of-plane scalars, namely $\bff{u}=(\widetilde{\bff{u}}, u_3)$, $\bff{B}=(\widetilde{\bff{B}}, B_3)$, $\bff{J}=(\widetilde{\bff{J}}, J_3)$, and $\bff{E}=(\widetilde{\bff{E}}, E_3)$.

We denote the in-plane gradient and in-plane Laplacian as $\nabla_\perp= (\pa_x,\pa_y)^\top$ and $\Delta_\perp= \nabla_\perp \cdot \nabla_\perp$, respectively. The 3D curl operator splits into two planar operators: the \emph{vector curl} $\nabla^\perp$ mapping scalar fields to in-plane vector fields, and the \emph{scalar curl} ($\mathrm{rot}$) mapping in-plane vector fields to scalars. For a scalar function $\phi:\mathscr{D}\subset \bb{R}^2 \to \bb{R}$ and an in-plane vector field $\widetilde{\bff{v}}=(v_1,v_2)$, these are defined by
\begin{align*}
	\nabla^\perp \phi:= (\pa_y \phi, -\pa_x \phi)^\top, \quad \mathrm{rot}\,\widetilde{\bff{v}}:=\pa_x v_2-\pa_y v_1.
\end{align*}
These operators are consistent with the 2D analogues of the cross product, namely
\begin{align*}
	\widetilde{\bff{v}}\times \phi:= (v_2 \phi, -v_1 \phi)^\top, \quad
	\widetilde{\bff{v}} \times \widetilde{\bff{w}} := v_1 w_2 - v_2 w_1.
\end{align*}

With these operators, the continuous 3D Hall--MHD system naturally decomposes into an in-plane system governing the transverse fields and an out-of-plane system describing scalar transport and magnetic stretching:
\begin{subequations}\label{equ:hall 25d}
	\begin{alignat}{1}
		(I-\alpha_1 \Delta_\perp) \pa_t \widetilde{\bff{u}} - \nu \Delta_\perp \widetilde{\bff{u}} + (\widetilde{\bff{u}} \cdot \nabla_\perp) \widetilde{\bff{u}} + \nabla_\perp p - \big(\widetilde{\bff{J}} \times B_3 - \widetilde{\bff{B}} \times J_3\big) &= \bff{0},
		\label{equ:hall 25d a}
		\\
		(I-\alpha_1 \Delta_\perp) \pa_t u_3 - \nu \Delta_\perp u_3 + (\widetilde{\bff{u}} \cdot \nabla_\perp)u_3 - (\widetilde{\bff{J}} \times \widetilde{\bff{B}}) &= 0,
		\label{equ:hall 25d b}
		\\
		\pa_t \widetilde{\bff{B}} + \nabla^\perp E_3 &= \bff{0},
		\label{equ:hall 25d c}
		\\
		\pa_t B_3 + \mathrm{rot}\, \widetilde{\bff{E}} &= 0,
		\label{equ:hall 25d d}
		\\
		\alpha_2 \pa_t \widetilde{\bff{J}}+ \sigma \widetilde{\bff{J}} + \eta \big(\widetilde{\bff{J}} \times B_3 - \widetilde{\bff{B}} \times J_3\big)
		-
		\widetilde{\bff{E}} - \big(\widetilde{\bff{u}} \times B_3 - \widetilde{\bff{B}} \times u_3\big) &= \bff{0},
		\label{equ:hall 25d e}
		\\
		\alpha_2 \pa_t J_3+ \sigma J_3+ \eta \widetilde{\bff{J}} \times \widetilde{\bff{B}} - E_3 - \widetilde{\bff{u}} \times \widetilde{\bff{B}} &= 0,
		\label{equ:hall 25d f}
		\\
		\widetilde{\bff{J}}- \nabla^\perp B_3 &= \bff{0},
		\label{equ:hall 25d g}
		\\
		J_3 - \mathrm{rot}\, \widetilde{\bff{B}} &= 0,
		\label{equ:hall 25d h}
		\\
		\nabla_\perp \cdot \widetilde{\bff{u}} = \nabla_\perp \cdot \widetilde{\bff{B}} &= 0,
		\label{equ:hall 25d i}
	\end{alignat}
\end{subequations}
subject to the initial data $\bff{u}(0)=\bff{u}_0$, $\bff{B}(0)=\bff{B}_0$, and boundary conditions:
\begin{align*}
	\widetilde{\bff{u}}= \bff{0}, \quad \widetilde{\bff{B}}\cdot \bff{n}= 0, \quad
	\widetilde{\bff{E}}\cdot \bff{\tau}=0, \quad
	\widetilde{\bff{J}} \cdot \bff{\tau}=0,
	\quad
	u_3 =E_3=J_3=0
	\quad \text{on } \pa \mathscr{D},
\end{align*}
where $\bff{n}=(n_1,n_2)^\top$ and $\bff{\tau}=(-n_2,n_1)^\top$ denote the unit outward normal and the unit tangent vector on $\pa \mathscr{D}$, respectively.

This structural decomposition is mirrored at the discrete level, where the discrete unknowns are sought in conforming finite element spaces:
\begin{align*}
	\bff{u}_h^n=(\widetilde{\bff{u}}_h^n, u_{3,h}^n)\in \bb{V}_h^0\times V_h^0, \quad
	p_h^n\in Q_h^0, \quad
	\bff{B}_h^n= (\widetilde{\bff{B}}_h^n, B_{3,h}^n)\in \bb{RT}_h^0 \times Q_h,
	\\
	\bff{E}_h^n= (\widetilde{\bff{E}}_h^n, E_{3,h}^n)\in \bb{X}_h^0 \times V_h^0, \quad
	\bff{J}_h^n= (\widetilde{\bff{J}}_h^n, J_{3,h}^n)\in \bb{X}_h^0 \times V_h^0.
\end{align*}
Furthermore, the 3D de Rham complex underlying the structure-preserving scheme reduces to the 2D rotated de Rham complex governing the in-plane magnetic field:
\begin{equation}\label{equ:exact cont 2d}
	\begin{tikzcd}
		H^1_0 \arrow[r, "\nabla^\perp"]  & \mathbb{H}_0(\mathrm{div}) \arrow[r, "\mathrm{div}"] & L^2_0.
	\end{tikzcd}
\end{equation}
In particular, when $\widetilde{\bff{B}}_h^n$ is sought in an $\mathbb{H}_0(\mathrm{div})$-conforming finite element space (e.g., Raviart--Thomas), its evolution is described by the in-plane induction equation analogous to \eqref{equ:fem reg mhd b}:
\[
\dtt \widetilde{\bff{B}}_h^n + \nabla^\perp E_{3,h}^n= \bff{0}.
\]
Since $\divg(\nabla^\perp \phi)\equiv 0$ holds, this update at the discrete level preserves the in-plane solenoidal constraint $\divg \widetilde{\bff{B}}_h^n =0$ exactly; cf.~\eqref{equ:fem reg mhd b} and \eqref{equ:dt B curl E} in the full 3D setting.

Furthermore, because this 2.5D formulation perfectly mirrors the antisymmetric coupling of the full 3D equations, it intrinsically inherits the energy dissipation law. To avoid redundancy, the rigorous numerical analysis and stability proofs will be detailed exclusively for the full 3D case, with the understanding that these theoretical guarantees carry over directly to the 2.5D numerical experiments presented in Section~\ref{subsec:tang}.

\section{Stability and convergence analysis}\label{sec:conv}

In this section, we analyse the stability and convergence of our scheme. First, we show that the scheme is well-posed and preserves the divergence of $\bff{B}_h^n$. Recall that the space $\bb{Y}_h$ was defined in Algorithm~\ref{alg:linear fem mhd}.

\begin{proposition}
	For $n=1,2,\ldots,N$, given $(\bff{u}_h^{n-1}, p_h^{n-1}, \bff{B}_h^{n-1}, \bff{E}_h^{n-1}, \bff{J}_h^{n-1})\in \bb{Y}_h$,
	there exists a unique $(\bff{u}_h^n, p_h^n, \bff{B}_h^n, \bff{E}_h^n, \bff{J}_h^n)\in \bb{Y}_h$ solving \eqref{equ:fem reg mhd}. Furthermore,
	\begin{equation}\label{equ:div zero disc}
		\divg \bff{B}_h^n=\divg \bff{B}_h^0.
	\end{equation}
	In particular, $\divg \bff{B}_h^n=0$ if $\divg \bff{B}_h^0=0$.
\end{proposition}

\begin{proof}
	To show the existence of a unique $(\bff{u}_h^n, p_h^n, \bff{B}_h^n, \bff{E}_h^n, \bff{J}_h^n)$ solving the linear system \eqref{equ:fem reg mhd}, it suffices to show that the corresponding homogeneous system: 
	\begin{subequations}\label{equ:sol reg mhd}
		\begin{alignat}{2}
			&\inpro{\bff{u}_h^n}{\bff{\phi}_h} + \alpha_1 \inpro{\nabla \bff{u}_h^n}{\nabla \bff{\phi}_h} 
			+ \tau\nu \inpro{\nabla \bff{u}_h^n}{\nabla \bff{\phi}_h}
			\; &&
			\nonumber\\
			&\qquad
			+
			\frac12 \tau \left[ \inpro{(\bff{u}_h^{n-1}\cdot \nabla) \bff{u}_h^n}{\bff{\phi}_h} - \inpro{(\bff{u}_h^{n-1}\cdot\nabla) \bff{\phi}_h}{\bff{u}_h^n} \right] 
			\; && 
			\nonumber\\
			&\qquad - \tau\inpro{p_h^n}{\divg \bff{\phi}_h} 
			- \tau \inpro{\bff{J}_h^n \times \bff{B}_h^{n-1}}{\bff{\phi}_h}
			= 
			0,
			\; && \quad \forall \bff{\phi}_h\in \bb{V}_{h}^0,
			\label{equ:sol reg mhd a}
			\\[1ex]
			&\inpro{\bff{B}_h^n}{\bff{\psi}_h} 
			+ \tau \inpro{\curl \bff{E}_h^n}{\bff{\psi}_h}= 0,
			\; && \quad \forall \bff{\psi}_h \in \bb{RT}_h^0,
			\label{equ:sol reg mhd b}
			\\[1ex]
			&\alpha_2 \inpro{\bff{J}_h^n}{\bff{\chi}_h}
			+ \tau \sigma \inpro{\bff{J}_h^n}{\bff{\chi}_h}
			+ \tau \eta \inpro{\bff{J}_h^n\times \bff{B}_h^{n-1}}{\bff{\chi}_h}
			\; &&
			\nonumber\\
			&\qquad
			= \tau \inpro{\bff{E}_h^n}{\bff{\chi}_h}
			+ \tau \inpro{\bff{u}_h^n\times \bff{B}_h^{n-1}}{\bff{\chi}_h},
			\; && \quad\forall \bff{\varphi}_h \in \bb{X}_h^0,
			\label{equ:sol reg mhd c}
			\\[1ex]
			&\inpro{\bff{J}_h^n}{\bff{\omega}_h}
			-\inpro{\bff{B}_h^n}{\curl \bff{\omega}_h}
			= 0,
			\; && \quad\forall \bff{\omega}_h \in \bb{X}_h^0,
			\label{equ:sol reg mhd e}
			\\[1ex]
			&\inpro{\divg\bff{u}_h^n}{q_h} = 0,
			\; && \quad\forall q_h\in Q_h.
			\label{equ:sol reg mhd f}
		\end{alignat}
	\end{subequations}
	has only the trivial solution, since $\bb{Y}_h$ is a finite-dimensional space. To this end, setting $\bff{\phi}_h= \bff{u}_h^n$, $\bff{\psi}_h= \bff{B}_h^n$, $\bff{\chi}_h= \bff{J}_h^n$, $\bff{\omega}_h= \tau \bff{E}_h^n$, and $q_h=\tau p_h^n$ in \eqref{equ:sol reg mhd}, and summing the resulting equations, we obtain
	\begin{align*}
		&\norm{\bff{u}_h^n}{\bb{L}^2}^2
		+
		\alpha_1 \norm{\nabla \bff{u}_h^n}{\bb{L}^2}^2
		+
		\tau\nu \norm{\nabla \bff{u}_h^n}{\bb{L}^2}^2
		+
		\norm{\bff{B}_h^n}{\bb{L}^2}^2
		+
		\alpha_2 \norm{\bff{J}_h^n}{\bb{L}^2}^2
		+
		\tau\sigma \norm{\bff{J}_h^n}{\bb{L}^2}^2
		= 0,
	\end{align*}
	from which we infer that $\bff{u}_h^n=\bff{B}_h^n=\bff{J}_h^n=\bff{0}$. By taking $\bff{\chi}_h=\bff{E}_h^n$ in \eqref{equ:sol reg mhd c}, we obtain $\bff{E}_h^n=\bff{0}$. From \eqref{equ:sol reg mhd a}, we also have $\inpro{p_h^n}{\divg \bff{\phi}_h}=0$ for all $\bff{\phi}_h\in \bb{V}_{h}^0$. By the discrete inf-sup condition \eqref{equ:inf sup} applied to $q_h=p_h^n$, we obtain
	\begin{align*}
		\beta \norm{p_h^n}{L^2}^2 \leq  \sup_{\bff{0}\neq \bff{\phi}_h\in \bb{V}_{h}^0} \frac{\inpro{\divg \bff{\phi}_h}{p_h^n}}{\norm{\nabla \bff{\phi}_h}{\bb{L}^2}}= 0,
	\end{align*}
	and thus $p_h^n=0$. This implies the existence of a unique~$(\bff{u}_h^n, p_h^n, \bff{B}_h^n, \bff{E}_h^n, \bff{J}_h^n)\in \bb{Y}_h$ solving \eqref{equ:fem reg mhd}.

	Finally, since $\mathrm{\curl}(\bb{X}_h^0)\subset \bb{RT}_h^0$ by \eqref{equ:exact fe}, equation \eqref{equ:fem reg mhd b} implies
	\begin{align}\label{equ:dt B curl E}
		\dtt \bff{B}_h^n+ \curl \bff{E}_h^n  =\bff{0}.
	\end{align}
	Taking the divergence of this equation gives $\divg \bff{B}_h^n= \divg \bff{B}_h^{n-1}$ pointwise a.e., thus implying \eqref{equ:div zero disc}, as required.
\end{proof}

At the continuous level, problem \eqref{equ:weak reg mhd} admits an energy dissipation law for the energy functional $\mathcal{E}$ given by
\begin{align}\label{equ:energy}
	\mathcal{E}[\bff{u},\bff{B}, \bff{J}]:= \frac12 \norm{\bff{u}}{\bb{L}^2}^2 + \frac{\alpha_1}{2} \norm{\nabla \bff{u}}{\bb{L}^2}^2 + \frac{1}{2} \norm{\bff{B}}{\bb{L}^2}^2 + \frac{\alpha_2}{2} \norm{\bff{J}}{\bb{L}^2}^2.
\end{align}
Next, we show that the proposed scheme satisfies a corresponding discrete energy law. In particular, the presence of the Voigt terms ($\alpha_1,\alpha_2>0$) yields enhanced stability control of $\nabla \bff{u}_h^n$ and $\bff{J}_h^n$.

\begin{proposition}
	Let $(\bff{u}_h^n, p_h^n, \bff{B}_h^n, \bff{E}_h^n, \bff{J}_h^n)$ be given by Algorithm~\ref{alg:linear fem mhd} and $\mathcal{E}$ be the energy functional defined in \eqref{equ:energy}. Then the following energy dissipation law holds unconditionally:
	\begin{align}\label{equ:ener diss}
		\mathcal{E}[\bff{u}_h^n, \bff{B}_h^n, \bff{J}_h^n] \leq \mathcal{E}[\bff{u}_h^{n-1}, \bff{B}_h^{n-1}, \bff{J}_h^{n-1}].
	\end{align}
	Furthermore, we have for $n=1,2,\ldots,N$,
	\begin{align}\label{equ:stab}
		&\mathcal{E}[\bff{u}_h^n, \bff{B}_h^n, \bff{J}_h^n]
		+
		\tau \nu \sum_{j=1}^n \norm{\nabla \bff{u}_h^j}{\bb{L}^2}^2 
		+
		\tau \sigma \sum_{j=1}^n \norm{\bff{J}_h^j}{\bb{L}^2}^2
		\leq 
		\mathcal{E}[\bff{u}_h^0, \bff{B}_h^0, \bff{J}_h^0].
	\end{align}
	and
	\begin{align}\label{equ:stab B L3}
		\alpha_2 \norm{\bff{B}_h^n}{\bb{L}^3}^2 
		+
		\tau \sigma\sum_{j=1}^n \norm{\bff{B}_h^j}{\bb{L}^3}^2
		\leq C,
	\end{align}
	where $C$ is independent of $\alpha_1$, $\alpha_2$, $n$, $h$, $\tau$, and $T$.
\end{proposition}

\begin{proof}
	First, note that by taking $\bff{\chi}_h= \bff{J}_h^n$, we have $\inpro{\bff{\zeta}_h^n}{\bff{J}_h^n}=0$.
	Next, taking $\bff{\phi}_h= \bff{u}_h^n$, $\bff{\psi}_h= \bff{B}_h^n$, $\bff{\chi}_h= \bff{J}_h^n$, $\bff{\omega}_h=\bff{E}_h^n$, and $q_h=p_h^n$ in \eqref{equ:fem reg mhd}, and summing the resulting equations, we obtain
	\begin{align*}
		&\frac{1}{2\tau} \left(\norm{\bff{u}_h^n}{\bb{L}^2}^2 - \norm{\bff{u}_h^{n-1}}{\bb{L}^2}^2 \right) 
		+
		\frac{1}{2\tau} \norm{\bff{u}_h^n-\bff{u}_h^{n-1}}{\bb{L}^2}^2
		\nonumber\\
		&\quad
		+
		\frac{\alpha_1}{2\tau} \left(\norm{\nabla \bff{u}_h^n}{\bb{L}^2}^2 - \norm{\nabla \bff{u}_h^{n-1}}{\bb{L}^2}^2 \right) 
		+
		\frac{\alpha_1}{2\tau} \norm{\nabla \bff{u}_h^n- \nabla \bff{u}_h^{n-1}}{\bb{L}^2}^2
		\nonumber\\
		&\quad
		+
		\frac{1}{2\tau} \left(\norm{\bff{B}_h^n}{\bb{L}^2}^2 - \norm{\bff{B}_h^{n-1}}{\bb{L}^2}^2 \right) 
		+
		\frac{1}{2\tau} \norm{\bff{B}_h^n-\bff{B}_h^{n-1}}{\bb{L}^2}^2
		\nonumber\\
		&\quad
		+
		\frac{\alpha_2}{2\tau} \left(\norm{\bff{J}_h^n}{\bb{L}^2}^2 - \norm{\bff{J}_h^{n-1}}{\bb{L}^2}^2 \right) 
		+
		\frac{\alpha_2}{2\tau} \norm{\bff{J}_h^n-\bff{J}_h^{n-1}}{\bb{L}^2}^2
		\nonumber\\
		&\quad
		+
		\nu \norm{\nabla \bff{u}_h^n}{\bb{L}^2}^2 + \sigma \norm{\bff{J}_h^n}{\bb{L}^2}^2
		= 0.
	\end{align*}
	Summing over $j\in \{1,2,\ldots,n\}$, we deduce \eqref{equ:stab}.
	
	Finally, \eqref{equ:fem reg mhd d} implies $\bff{J}_h^n= \curlh \bff{B}_h^n$. Therefore, since $\divg \bff{B}_h^n=0$ by \eqref{equ:div zero disc}, we have by the generalised Gaffney inequality \eqref{equ:gaffney},
	\[
	\norm{\bff{B}_h^n}{\bb{L}^3}\leq \norm{\curlh \bff{B}_h^n}{\bb{L}^2}= \norm{\bff{J}_h^n}{\bb{L}^2}.
	\]
	The estimate \eqref{equ:stab B L3} then follows from \eqref{equ:stab}, thus completing the proof of the proposition.
\end{proof}

We aim to prove error estimates for our scheme. To this end, we utilise the following error decompositions:
\begin{align}
	\label{equ:uh error decomp}
	&\bff{u}_h^n-\bff{u}^n= (\bff{u}_h^n-\mathcal{S}_h\bff{u}^n) + (	\mathcal{S}_h\bff{u}^n- \bff{u}^n) =: \thetau^n + \rhou^n,
	\\
	\label{equ:ph error decomp}
	&p_h^n-p^n= (p_h^n-\mathcal{S}_h p^n) + (	\mathcal{S}_h p^n- p^n) =: \thetap^n + \rhop^n,
	\\
	\label{equ:Bh error decomp}
	&\bff{B}_h^n-\bff{B}^n= (\bff{B}_h^n-\Pi_h^{\bb{RT}} \bff{B}^n) + (	\Pi_h^{\bb{RT}} \bff{B}^n- \bff{B}^n) =: \thetab^n + \rhob^n,
	\\
	\label{equ:Eh error decomp}
	&\bff{E}_h^n-\bff{E}^n= (\bff{E}_h^n-\mathcal{I}_h^{\bb{X}}\bff{E}^n) + (	\mathcal{I}_h^{\bb{X}}\bff{E}^n- \bff{E}^n) =: \thetae^n + \rhoe^n,
	\\
	\label{equ:Jh error decomp}
	&\bff{J}_h^n-\bff{J}^n= (\bff{J}_h^n-\Pi_h^{\bb{X}}\bff{J}^n) + (\Pi_h^{\bb{X}}\bff{J}^n- \bff{J}^n) =: \thetaj^n + \rhoj^n,
\end{align}
where $\mathcal{S}_h$ is the Stokes projection defined in \eqref{equ:stokes proj}, $\mathcal{I}_h^{\bb{X}}$ is the quasi-interpolator satisfying the commutative diagram \eqref{equ:comm}, while $\Pi_h^{\bb{RT}}$ and $\Pi_h^{\bb{X}}$ are the orthogonal projections defined in~\eqref{equ:Pi h L2}. Consequently, we note that
\begin{align}\label{equ:Pi zero}
	\inpro{\rhob^n}{\bff{\psi}_h}=\inpro{\rhoj^n}{\bff{\chi}_h}=0,
\end{align}
for any $\bff{\psi}_h\in \bb{RT}_h^0$, $\bff{\chi}_h\in \bb{X}_h^0$.

With the above error decompositions, by subtracting \eqref{equ:weak reg mhd} from \eqref{equ:fem reg mhd}, noting \eqref{equ:stokes proj} and \eqref{equ:Pi zero}, we have the following key error equations:
\begin{subequations}\label{equ:error equ}
	\begin{alignat}{2}
		&\inpro{\dtt \thetau^n+\dtt \rhou^n+\dtt \bff{u}^n-\partial_t \bff{u}^n}{\bff{\phi}_h} 
		\; &&
		\nonumber\\
		&\;
		+ \alpha_1 \inpro{\nabla\dtt \thetau^n+\nabla\dtt \rhou^n+ \nabla\dtt \bff{u}^n-\nabla\partial_t \bff{u}^n}{\nabla \bff{\phi}_h} 
		\; &&
		\nonumber\\
		&\;		
		+ \nu \inpro{\nabla \thetau^n}{\nabla \bff{\phi}_h}
		+ \frac12 \left[ \inpro{(\bff{u}_h^{n-1}\cdot\nabla)(\thetau^n+\rhou^n)}{\bff{\phi}_h}\right] 
		\; &&
		\nonumber\\
		&\;
		+
		\frac12 \left[\inpro{\big( (\thetau^{n-1}+\rhou^{n-1}+\bff{u}^{n-1}-\bff{u}^n)\cdot \nabla\big) \bff{u}^n}{\bff{\phi}_h}\right]
		\; &&
		\nonumber\\
		&\;
		-
		\frac12 \left[ \inpro{(\bff{u}_h^{n-1}\cdot\nabla) \bff{\phi}_h}{\thetau^n+\rhou^n} \right] 
		\; &&
		\nonumber\\
		&\;
		+ \frac12 \left[\inpro{\big( (\thetau^{n-1}+\rhou^{n-1}+\bff{u}^{n-1}-\bff{u}^n)\cdot \nabla\big) \bff{\phi}_h}{\bff{u}^n} \right] 
		\; && 
		\nonumber\\
		&\; 
		- \inpro{\thetap^n}{\divg \bff{\phi}_h} 
		- \inpro{(\thetaj^n+\rhoj^n)\times \bff{B}_h^{n-1}}{\bff{\phi}_h}
		\; &&
		\nonumber\\
		&\;
		- \inpro{\bff{J}^n\times (\thetab^{n-1}+\rhob^{n-1} +\bff{B}^{n-1}-\bff{B}^n)}{\bff{\phi}_h}
		= 
		0,
		\; && \; \forall \bff{\phi}_h\in \bb{V}_{h}^0,
		\label{equ:error equ a}
		\\[1ex]
		&\inpro{\dtt \thetab^n+ \dtt\bff{B}^n- \partial_t \bff{B}^n}{\bff{\psi}_h}
		+ \inpro{\curl \thetae^n+ \curl \rhoe^n}{\bff{\psi}_h} = 0,
		\; && \; \forall \bff{\psi}_h \in \bb{RT}_h^0,
		\label{equ:error equ b}
		\\[1ex]
		&\alpha_2 \inpro{\dtt \thetaj^n +\dtt \rhoj^n+ \dtt \bff{J}^n- \partial_t \bff{J}^n}{\bff{\chi}_h}
		+ \sigma\inpro{\thetaj^n+ \rhoj^n}{\bff{\chi}_h}
		\; &&
		\nonumber\\
		&\;
		+ \eta \inpro{\bff{J}_h^n\times (\thetab^{n-1}+\rhob^{n-1}+\bff{B}^{n-1}-\bff{B}^n)}{\bff{\chi}_h}
		\; &&
		\nonumber\\
		&\;
		+ \eta \inpro{(\thetaj^n+\rhoj^n) \times \bff{B}^n}{\bff{\chi}_h}
		\; &&
		\nonumber\\
		&
		= \inpro{\thetae^n+\rhoe^n}{\bff{\chi}_h}
		+ \inpro{\bff{u}_h^n\times (\thetab^{n-1}+\rhob^{n-1}+\bff{B}^{n-1}-\bff{B}^n)}{\bff{\chi}_h}
		\; &&
		\nonumber\\
		&\;
		+ \inpro{(\thetau^n+\rhou^n)\times \bff{B}^n}{\bff{\chi}_h},
		\; && \;\forall \bff{\chi}_h \in \bb{X}_h^0,
		\label{equ:error equ c}
		\\[1ex]
		&\inpro{\thetaj^n+\rhoj^n}{\bff{\omega}_h}
		-\inpro{\thetab^n+\rhob^n}{\curl \bff{\omega}_h}
		= 0,
		\; && \;\forall \bff{\omega}_h \in \bb{X}_h^0,
		\label{equ:error equ d}
		\\[1ex]
		&\inpro{\divg\thetau^n}{q_h} = 0,
		\; && \;\forall q_h\in Q_h.
		\label{equ:error equ e}
	\end{alignat}
\end{subequations}

Using the above error equations, we now prove an auxiliary error estimate that will play a key role in the proof of the main theorem. In the proof, we often use the following standard inequalities, which follow from Taylor's theorem: If $\bff{v}\in W^{1,\infty}_T(\bb{L}^p)$ for some $p\in [1,\infty]$, then 
\begin{align}\label{equ:dtt vn Lp}
	\norm{\dtt \bff{v}^n}{\bb{L}^p} &\leq \frac{1}{\tau} \int_{t_{n-1}}^{t_n} \norm{\pa_t \bff{v}(s)}{\bb{L}^p} \ds 
	\leq \norm{\bff{v}}{W^{1,\infty}_T(\bb{L}^p)}.
\end{align}
If $\bff{v}\in H^2_T(\bb{L}^p)$, then
\begin{align}\label{equ:dtt vn min pa t}
	\tau \sum_{m=1}^n \norm{\dtt \bff{v}^m- \pa_t \bff{v}^m}{\bb{L}^p}^2
	&\leq
	C\tau^2 \norm{\partial_{tt} \bff{v}}{L^2_T(\bb{L}^p)}^2.
\end{align}

\begin{proposition}\label{pro:aux}
	Let $(\bff{u}_h^n, p_h^n, \bff{B}_h^n, \bff{E}_h^n, \bff{J}_h^n)$ be given by Algorithm~\ref{alg:linear fem mhd}, and let $(\bff{u}, p ,\bff{B},\bff{E}, \bff{J})$ be the solution of \eqref{equ:weak reg mhd} with regularity given by \eqref{equ:assum}. For $n\in \{1,2,\ldots,\lfloor T/\tau \rfloor\}$,
	\begin{align}\label{equ:err aux}
		&\norm{\thetau^n}{\bb{L}^2}^2 + \alpha_1 \norm{\nabla \thetau}{\bb{L}^2}^2 + \norm{\thetab^n}{\bb{L}^2}^2 + 
		\alpha_2 \norm{\thetaj^n}{\bb{L}^2}^2
		\nonumber\\
		&\qquad
		+
		\tau \sum_{j=1}^n \left(\norm{\nabla \thetau^j}{\bb{L}^2}^2 + \norm{\thetaj^n}{\bb{L}^2}^2 \right)
		\leq
		C(h^{2s}+\tau^2).
	\end{align}
	If $\sigma>0$, then the constant $C$ depends on $T$ and possibly on $\alpha_1,\alpha_2$, but is independent of $n$, $h$, $\tau$, and remains bounded as $\alpha_1,\alpha_2\to 0$. If $\sigma=0$, then \eqref{equ:err aux} remains valid, with $C$ depending on $T$ and $\alpha_2^{-1}$.
\end{proposition}

\begin{proof}
	First, we choose $\bff{\phi}_h= \thetau^n$ and invoke \eqref{equ:stokes proj} to obtain, after suitable cancellations,
	\begin{align}\label{equ:err sub u}
		&\frac{1}{2\tau} \left(\norm{\thetau^n}{\bb{L}^2}^2- \norm{\thetau^{n-1}}{\bb{L}^2}^2 \right)
		+
		\frac{1}{2\tau} \norm{\thetau^n-\thetau^{n-1}}{\bb{L}^2}^2
		\nonumber\\
		&\quad
		+
		\frac{\alpha_1}{2\tau} \left(\norm{\nabla \thetau^n}{\bb{L}^2}^2- \norm{\nabla \thetau^{n-1}}{\bb{L}^2}^2 \right)
		+
		\frac{\alpha_1}{2\tau} \norm{\nabla \thetau^n- \nabla \thetau^{n-1}}{\bb{L}^2}^2
		+
		\nu \norm{\nabla \thetau^n}{\bb{L}^2}^2
		\nonumber\\
		&=
		-\inpro{\dtt \rhou^n}{\thetau^n}
		- \inpro{\dtt \bff{u}^n-\pa_t\bff{u}^n}{\thetau^n}
		\nonumber\\
		&\quad
		- \alpha_1\inpro{\nabla\dtt \rhou^n}{\nabla\thetau^n}
		- \alpha_1 \inpro{\nabla\dtt \bff{u}^n- \nabla\pa_t\bff{u}^n}{\nabla\thetau^n}
		\nonumber\\
		&\quad
		-
		\frac12 \Big[\inpro{(\bff{u}_h^{n-1}\cdot\nabla) \rhou^n}{\thetau^n}
		-
		\inpro{(\bff{u}_h^{n-1}\cdot\nabla)\thetau^n}{\rhou^n} \Big] 
		\nonumber\\
		&\quad
		-
		\frac12 \Big[ \inpro{\big( (\thetau^{n-1}+\rhou^{n-1}+\bff{u}^{n-1}-\bff{u}^n)\cdot \nabla\big) \bff{u}^n}{\thetau^n} \Big] 
		\nonumber\\
		&\quad
		- 
		\frac12 \Big[
		\inpro{\big( (\thetau^{n-1}+\rhou^{n-1}+\bff{u}^{n-1}-\bff{u}^n)\cdot \nabla\big) \thetau^n}{\bff{u}^n}\Big]
		\nonumber\\
		&\quad
		+ \inpro{\thetap^n}{\divg \thetau^n} 
		+ \inpro{(\thetaj^n+\rhoj^n)\times \bff{B}_h^{n-1}}{\thetau^n}
		\nonumber\\
		&\quad
		+ \inpro{\bff{J}^n\times (\thetab^{n-1}+\rhob^{n-1} +\bff{B}^{n-1}-\bff{B}^n)}{\thetau^n}.
	\end{align}
	Second, we set $\bff{\psi}_h=\thetab^n$ to have
	\begin{align}\label{equ:err sub B}
		&\frac{1}{2\tau} \left(\norm{\thetab^n}{\bb{L}^2}^2- \norm{\thetab^{n-1}}{\bb{L}^2}^2 \right)
		+
		\frac{1}{2\tau} \norm{\thetab^n-\thetab^{n-1}}{\bb{L}^2}^2
		\nonumber\\
		&=
		-\inpro{\dtt \bff{B}^n-\pa_t \bff{B}^n}{\thetab^n}
		-\inpro{\curl\thetae^n + \curl \rhoe^n}{\thetab^n}.
	\end{align}
	Next, we put $\bff{\chi}_h=\thetaj^n$ and note the identity $(\bff{a}\times \bff{b})\cdot \bff{a}=0$ for any $\bff{a},\bff{b}\in \bb{R}^3$ to obtain
	\begin{align}\label{equ:err sub J}
		&\frac{\alpha_2}{2\tau} \left(\norm{\thetaj^n}{\bb{L}^2}^2- \norm{\thetaj^{n-1}}{\bb{L}^2}^2 \right)
		+
		\frac{\alpha_2}{2\tau} \norm{\thetaj^n-\thetaj^{n-1}}{\bb{L}^2}^2
		+
		\sigma \norm{\thetaj^n}{\bb{L}^2}^2
		\nonumber\\
		&=
		-\alpha_2 \inpro{\dtt \bff{J}^n-\pa_t \bff{J}^n}{\thetaj^n}
		+
		\inpro{\thetae^n+\rhoe^n}{\thetaj^n}
		\nonumber\\
		&\quad
		- \eta \inpro{\bff{J}_h^n\times (\thetab^{n-1}+\rhob^{n-1}+\bff{B}^{n-1}-\bff{B}^n)}{\thetaj^n}
		- \eta \inpro{\rhoj^n \times \bff{B}^n}{\thetaj^n}
		\nonumber\\
		&\quad
		+
		\inpro{(\thetau^n+\rhou^n)\times \bff{B}_h^{n-1}}{\thetaj^n}
		+
		\inpro{\bff{u}^n\times (\thetab^{n-1}+\rhob^{n-1}+\bff{B}^{n-1}-\bff{B}^n)}{\thetaj^n}.
	\end{align}
	Finally, setting $\bff{\omega}_h = \thetae^n$, we obtain
	\begin{align}\label{equ:err sub E}
		\inpro{\thetaj^n}{\thetae^n}
		&=
		\inpro{\thetab^n}{\curl \thetae^n},
	\end{align}
	where we used \eqref{equ:Pi zero} and the fact that $\curl \thetae^n \in \bb{RT}_h^0$.
	On the other hand, setting $q_h = \thetap^n$ and noting~\eqref{equ:stokes proj}, we have
	\begin{align}\label{equ:err sub p}
		\inpro{\divg \thetau^n}{\thetap^n} = 0.
	\end{align}
	We now add \eqref{equ:err sub u}, \eqref{equ:err sub B}, \eqref{equ:err sub J}, \eqref{equ:err sub E}, and \eqref{equ:err sub p} to obtain, after appropriate cancellations,
	\begin{align}\label{equ:sub to est}
		&\frac{1}{2\tau} \left(\norm{\thetau^n}{\bb{L}^2}^2- \norm{\thetau^{n-1}}{\bb{L}^2}^2 \right)
		+
		\frac{1}{2\tau} \norm{\thetau^n-\thetau^{n-1}}{\bb{L}^2}^2
		\nonumber\\
		&\quad
		+
		\frac{\alpha_1}{2\tau} \left(\norm{\nabla \thetau^n}{\bb{L}^2}^2- \norm{\nabla \thetau^{n-1}}{\bb{L}^2}^2 \right)
		+
		\frac{\alpha_1}{2\tau} \norm{\nabla \thetau^n- \nabla \thetau^{n-1}}{\bb{L}^2}^2
		+
		\nu \norm{\nabla \thetau^n}{\bb{L}^2}^2
		\nonumber\\
		&\quad
		+
		\frac{1}{2\tau} \left(\norm{\thetab^n}{\bb{L}^2}^2- \norm{\thetab^{n-1}}{\bb{L}^2}^2 \right)
		+
		\frac{1}{2\tau} \norm{\thetab^n-\thetab^{n-1}}{\bb{L}^2}^2
		\nonumber\\
		&\quad
		+
		\frac{\alpha_2}{2\tau} \left(\norm{\thetaj^n}{\bb{L}^2}^2- \norm{\thetaj^{n-1}}{\bb{L}^2}^2 \right)
		+
		\frac{\alpha_2}{2\tau} \norm{\thetaj^n-\thetaj^{n-1}}{\bb{L}^2}^2
		+
		\sigma \norm{\thetaj^n}{\bb{L}^2}^2
		\nonumber\\
		&=
		-\inpro{\dtt \rhou^n}{\thetau^n}
		- \inpro{\dtt \bff{u}^n-\pa_t\bff{u}^n}{\thetau^n}
		\nonumber\\
		&\quad
		- \alpha_1\inpro{\nabla\dtt \rhou^n}{\nabla\thetau^n}
		- \alpha_1 \inpro{\nabla\dtt \bff{u}^n- \nabla\pa_t\bff{u}^n}{\nabla\thetau^n}
		\nonumber\\
		&\quad
		-
		\frac12 \Big[\inpro{(\bff{u}_h^{n-1}\cdot\nabla) \rhou^n}{\thetau^n}
		-
		\inpro{(\bff{u}_h^{n-1}\cdot\nabla)\thetau^n}{\rhou^n} \Big] 
		\nonumber\\
		&\quad
		-
		\frac12 \Big[ \inpro{\big( (\thetau^{n-1}+\rhou^{n-1}+\bff{u}^{n-1}-\bff{u}^n)\cdot \nabla\big) \bff{u}^n}{\thetau^n} \Big] 
		\nonumber\\
		&\quad
		-
		\frac12 \Big[
		\inpro{\big( (\thetau^{n-1}+\rhou^{n-1}+\bff{u}^{n-1}-\bff{u}^n)\cdot \nabla\big) \thetau^n}{\bff{u}^n}\Big]
		\nonumber\\
		&\quad
		+ \inpro{\rhoj^n\times \bff{B}_h^{n-1}}{\thetau^n}
		+ \inpro{\bff{J}^n\times (\thetab^{n-1}+\rhob^{n-1} +\bff{B}^{n-1}-\bff{B}^n)}{\thetau^n}
		\nonumber\\
		&\quad
		-\inpro{\dtt \bff{B}^n-\pa_t \bff{B}^n}{\thetab^n}
		-\inpro{\curl \rhoe^n}{\thetab^n}
		\nonumber\\
		&\quad
		-\alpha_2 \inpro{\dtt \bff{J}^n-\pa_t \bff{J}^n}{\thetaj^n}
		+
		\inpro{\rhoe^n}{\thetaj^n}
		\nonumber\\
		&\quad
		- \eta \inpro{\bff{J}_h^n\times (\thetab^{n-1}+\rhob^{n-1}+\bff{B}^{n-1}-\bff{B}^n)}{\thetaj^n}
		- \eta \inpro{\rhoj^n \times \bff{B}^n}{\thetaj^n}
		\nonumber\\
		&\quad
		+
		\inpro{\rhou^n\times \bff{B}_h^{n-1}}{\thetaj^n}
		+
		\inpro{\bff{u}^n\times (\thetab^{n-1}+\rhob^{n-1}+\bff{B}^{n-1}-\bff{B}^n)}{\thetaj^n}
		\nonumber\\
		&=: I_1+I_2+\ldots+I_{16}.
	\end{align}
	It remains to estimate each term $I_j$, $j=1,2,\ldots, 16$, on the last line. Let $\epsilon>0$ be a sufficiently small number to be fixed later. First, the terms $I_1$ to $I_4$ are bounded using \eqref{equ:dtt vn Lp}, \eqref{equ:Stokes approx}, and Young's inequality to obtain
	\begin{align*}
		\abs{I_1}
		&\leq
		Ch^{2s}+ \epsilon \norm{\thetau^n}{\bb{L}^2}^2,
		\\
		\abs{I_2}
		&\leq
		C \norm{\dtt \bff{u}^n- \pa_t \bff{u}^n}{\bb{L}^2}^2
		+
		\epsilon \norm{\thetau^n}{\bb{L}^2}^2,
		\\
		\abs{I_3}
		&\leq
		C\alpha_1^2 h^{2s}+ \frac{\nu}{4} \norm{\nabla \thetau^n}{\bb{L}^2}^2,
		\\
		\abs{I_4}
		&\leq
		C\alpha_1^2 \norm{\nabla\dtt \bff{u}^n- \nabla\pa_t \bff{u}^n}{\bb{L}^2}^2
		+
		\frac{\nu}{4} \norm{\nabla \thetau^n}{\bb{L}^2}^2.
	\end{align*}
	For $I_5$, we employ Young's inequality and the Gagliardo--Nirenberg inequality to obtain
	\begin{align*}
		\abs{I_5}
		&\leq
		\frac12 \norm{\bff{u}_h^{n-1}}{\bb{L}^6} \norm{\nabla\rhou^n}{\bb{L}^2}
		\norm{\thetau^n}{\bb{L}^3}
		+
		\frac12 \norm{\bff{u}_h^{n-1}}{\bb{L}^6}
		\norm{\nabla \thetau^n}{\bb{L}^2}
		\norm{\rhou^n}{\bb{L}^3}
		\\
		&\leq
		Ch^{2s} \norm{\nabla \bff{u}_h^{n-1}}{\bb{L}^2}^2 +
		\epsilon \norm{\nabla \thetau^n}{\bb{L}^2}^2,
	\end{align*}
	where in the last step we also used \eqref{equ:Stokes approx}, the Sobolev embedding, and the Poincar\'e inequality. Similarly, for $I_6$, by \eqref{equ:dtt vn Lp}, H\"older's and Young's inequality, noting \eqref{equ:Stokes approx} and the regularity of the solution, we have
	\begin{align*}
		\abs{I_6}
		&\leq
		\frac12 \big(\norm{\thetau^{n-1}}{\bb{L}^2} + \norm{\rhou^{n-1}}{\bb{L}^2} + \norm{\bff{u}^{n-1}-\bff{u}^n}{\bb{L}^2} \big) 
		\\
		&\qquad 
		\big(\norm{\nabla \bff{u}^n}{\bb{L}^3} \norm{\thetau^n}{\bb{L}^6} + \norm{\nabla \thetau^n}{\bb{L}^2} \norm{\bff{u}^n}{\bb{L}^\infty} \big) 
		\\
		&\leq
		Ch^{2s}+C\tau^2+ C\norm{\thetau^{n-1}}{\bb{L}^2}^2
		+
		\epsilon \norm{\nabla\thetau^n}{\bb{L}^2}^2.
	\end{align*}
	For the term $I_7$, by Young's and Poincar\'e's inequality, and \eqref{equ:PiX approx}, we have
	\begin{align*}
		\abs{I_7}
		&\leq
		C\norm{\rhoj^n}{\bb{L}^2}^2 \norm{\bff{B}_h^{n-1}}{\bb{L}^3}^2 + \epsilon \norm{\thetau}{\bb{L}^6}^2
		\leq
		Ch^{2s} \norm{\bff{B}_h^{n-1}}{\bb{L}^3}^2 + \epsilon \norm{\nabla \thetau}{\bb{L}^2}^2.
	\end{align*}
	Next, for $I_8$, by Young's and Poincar\'e's inequality, \eqref{equ:PiRT approx}, and \eqref{equ:dtt vn Lp}, we have
	\begin{align*}
		\abs{I_8}
		&\leq
		C\norm{\bff{J}^n}{\bb{L}^3}^2 \norm{\thetab^{n-1}+\rhob^{n-1}+ \bff{B}^{n-1}-\bff{B}^n}{\bb{L}^2}^2 + \epsilon \norm{\thetau^n}{\bb{L}^6}^2
		\\
		&\leq
		Ch^{2s}+C\tau^2+ C\norm{\thetab^{n-1}}{\bb{L}^2}^2 + \epsilon \norm{\nabla \thetau^n}{\bb{L}^2}^2.
	\end{align*}
	The terms $I_9$ and $I_{11}$ are estimated in a similar manner as $I_2$, leading to
	\begin{align*}
		\abs{I_9}
		&\leq
		C \norm{\dtt \bff{B}^n- \pa_t \bff{B}^n}{\bb{L}^2}^2
		+
		\epsilon \norm{\thetab^n}{\bb{L}^2}^2,
		\\
		\abs{I_{11}}
		&\leq
		C \alpha_2^2 \norm{\dtt \bff{J}^n- \pa_t \bff{J}^n}{\bb{L}^2}^2
		+
		\frac{\sigma}{4} \norm{\thetaj^n}{\bb{L}^2}^2.
	\end{align*}
	Next, we estimate $I_{10}$. Note that by the commutative diagram property \eqref{equ:comm},
	\begin{align*}
		\curl \rhoe^n= \curl (\mathcal{I}_h^{\bb{X}} \bff{E}^n- \bff{E}^n)= \mathcal{I}_h^{\bb{RT}} (\curl \bff{E}^n)- \curl \bff{E}^n.
	\end{align*}
	Therefore, by \eqref{equ:interp RT approx} and Young's inequality,
	\begin{align*}
		\abs{I_{10}}
		&\leq
		Ch^{2s} + \epsilon\norm{\thetab^n}{\bb{L}^2}^2.
	\end{align*}
	Similarly for $I_{12}$, by \eqref{equ:interp ned approx} and Young's inequality,
	\begin{align*}
		\abs{I_{12}}
		&\leq
		Ch^{2s} + \epsilon\norm{\thetaj^n}{\bb{L}^2}^2.
	\end{align*}
	For $I_{13}$, we write $\bff{J}_h^n=\thetaj^n+ \Pi_h^{\bb{X}} \bff{J}^n$, then apply \eqref{equ:proj Pi stab}, \eqref{equ:PiX approx}, \eqref{equ:dtt vn Lp}, and Young's inequality to obtain
	\begin{align*}
		\abs{I_{13}}
		&\leq
		C \norm{\Pi_h^{\bb{X}} \bff{J}^n}{\bb{L}^\infty} \norm{\thetab^{n-1}+\rhob^{n-1}+\bff{B}^{n-1}-\bff{B}^n}{\bb{L}^2} \norm{\thetaj^n}{\bb{L}^2}
		\\
		&\leq
		Ch^{2s}+C\tau^2 + C\norm{\thetab^{n-1}}{\bb{L}^2}^2 + \epsilon \norm{\thetaj^n}{\bb{L}^2}^2.
	\end{align*}
	For $I_{14}$ and $I_{15}$, by \eqref{equ:PiX approx}, \eqref{equ:Stokes approx}, and Young's inequality, we have
	\begin{align*}
		\abs{I_{14}}+ \abs{I_{15}}
		&\leq
		C\norm{\rhoj^n}{\bb{L}^2} \norm{\bff{B}^n}{\bb{L}^\infty} \norm{\thetaj^n}{\bb{L}^2} + C\norm{\rhou^n}{\bb{L}^6} \norm{\bff{B}_h^{n-1}}{\bb{L}^3} \norm{\thetaj^n}{\bb{L}^2}
		\\
		&\leq
		Ch^{2s} \left(1+ \norm{\bff{B}_h^{n-1}}{\bb{L}^3}^2\right)
		+ \epsilon \norm{\thetaj^n}{\bb{L}^2}^2.
	\end{align*}
	Finally, by a similar argument, we obtain
	\begin{align*}
		\abs{I_{16}}
		&\leq
		Ch^{2s}+C\tau^2 + C\norm{\thetab^{n-1}}{\bb{L}^2}^2 + \epsilon \norm{\thetaj^n}{\bb{L}^2}^2.
	\end{align*}
	We now substitute these estimates back into \eqref{equ:sub to est}. If $\sigma>0$, then we can choose $\epsilon>0$ sufficiently small to absorb the relevant terms, and sum over $m\in \{1,2,\ldots,n\}$. Applying the discrete Gronwall lemma, noting \eqref{equ:stab B L3} and \eqref{equ:dtt vn min pa t}, we obtain~\eqref{equ:err aux}.
	
	If $\sigma=0$, the term $\epsilon \norm{\thetaj^n}{\bb{L}^2}^2$ arising in the estimates of $I_{12}$ up to $I_{16}$ can no longer be absorbed by the dissipative contribution $\sigma \norm{\thetaj^n}{\bb{L}^2}^2$ in~\eqref{equ:sub to est}. Instead, we make use of the term with coefficient $\alpha_2$ on the left-hand side of~\eqref{equ:sub to est} to control $\norm{\thetaj^n}{\bb{L}^2}^2$, and apply the discrete Gronwall lemma with sufficiently small $\tau>0$. Consequently, the constant $C$ in \eqref{equ:err aux} depends on $\alpha_2^{-1}$ in this case. The proof is now complete.
\end{proof}

With the preceding estimates established, we are now in a position to prove our main result, Theorem~\ref{the:main}.

\begin{proof}[Proof of Theorem~\ref{the:main}]
	This follows from Proposition~\ref{pro:aux}, the error decompositions \eqref{equ:uh error decomp}--\eqref{equ:Jh error decomp}, and the triangle inequality.
\end{proof}

\section{Numerical simulations}\label{sec:num sim}

We perform several physically relevant numerical simulations to assess the performance of our scheme.
To verify convergence, we compute a reference solution on a fine mesh and with a small time step. For any variable $\bff{v}$, the error $\mathcal{E}_{s}^{\bff{v}}(h,\tau)$ at a fixed time $T$ is defined by
\[
\mathcal{E}_{s}^{\bff{v}}(h,\tau):= \norm{\bff{v}_h^N- \bff{v}_{\text{ref}}(T)}{\bb{H}^s}, \quad s\in \{0,1\}.
\]
Here, $\bff{v}_h^N$ is the numerical solution with mesh size $h$ at time $T=N\tau$, and $\bff{v}_{\text{ref}}(T)$ denotes the reference solution at the same time $T$.

\subsection{The confined Arnold--Beltrami--Childress flow}\label{subsec:abc}

Fix $\mathscr{D}=[0,1]^3$. We simulate the evolution of a magnetic field embedded in an Arnold--Beltrami--Childress (ABC) flow. This configuration is a prototypical flow where the velocity field $\bff{u}$ is perfectly aligned with its vorticity $\curl \bff{u}$, creating a chaotic system of interlocking helical streamlines.
The initial fluid velocity is set to be 
\[
\bff{u}_0(x,y,z)= \big(\sin(2\pi y)\cos(\pi z), \sin (\pi y)\cos(\pi x), \sin(\pi x)\cos(\pi y)\big).
\]
The initial magnetic field is defined as $\bff{B}_0=\curl \bff{A}_0$, where
\[
\bff{A}_0(x,y,z)= \big(\sin(2\pi y)\sin(\pi z), \sin (2\pi y)\sin(\pi x), \sin(2\pi x)\sin(\pi y)\big).
\]
This verifies $\divg \bff{B}_0=0$ in $\mathscr{D}$ and $\bff{B}_0\cdot \bff{n}=0$ on $\pa \mathscr{D}$.
The parameters of the problem are $\nu=\sigma=0.005$, $\eta=0.1$, and $\alpha_1=\alpha_2=10^{-5}$.

Snapshots of $\bff{u},\bff{B}$, and $\bff{J}$ on two perpendicular slices ($z=0.5$ and $x=0.5$) of $\mathscr{D}$ at selected times are shown in Figure~\ref{fig:snapshots exp 1}. 
As the initial smooth channels of the magnetic field are twisted by the fluid, magnetic tension builds until a tearing instability occurs. This leads to the fragmentation into discrete magnetic islands and flux ropes. The Hall term triggers the formation of patches of negative $B_z$ components between dominant positive $B_z$ islands. In the long term, the system undergoes an inverse energy cascade where smaller islands coalesce into larger structures. The current density and fluid velocity decay as the system reaches a relaxed state. Despite the relatively low resolution ($N=16$), this energy-stable scheme preserves the solenoidal constraint on $\bff{B}$ (up to solver tolerance) and resolves the main features accurately; see Figure~\ref{fig:energy div exp1}.

Finally, we take the reference solution at $T=0.02$ with $h=1/20$ and $\tau=0.005$. Plots of the errors of $\bff{u},\bff{B},\bff{J}$ against $1/h$ to verify spatial orders of convergence at $T=0.02$ are shown in Figures~\ref{fig:conv u 1} and~\ref{fig:conv b j 0}. The observed convergence rates for $\bff{u}$ and $\bff{J}$ are reduced, which may be attributed to limited regularity of the solution. Now, with the same $T$, we take the reference solution with $h=1/20$ and $\tau=0.001$. Plots of $\mathcal{E}_0^{\bff{u}}$ and $\mathcal{E}_0^{\bff{B}}$ against $1/\tau$ are displayed in Figure~\ref{fig:temp conv u b 0}.

\begin{figure}[!htb]
	\centering
	\begin{subfigure}[b]{0.45\textwidth}
		\centering
		\includegraphics[width=\textwidth]{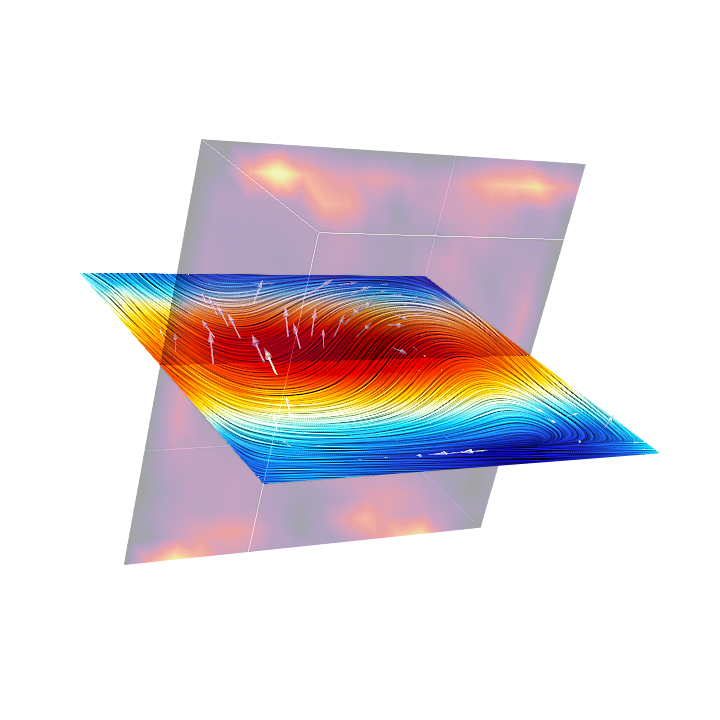}
		\caption{$t=0$}
	\end{subfigure}
	\begin{subfigure}[b]{0.45\textwidth}
		\centering
		\includegraphics[width=\textwidth]{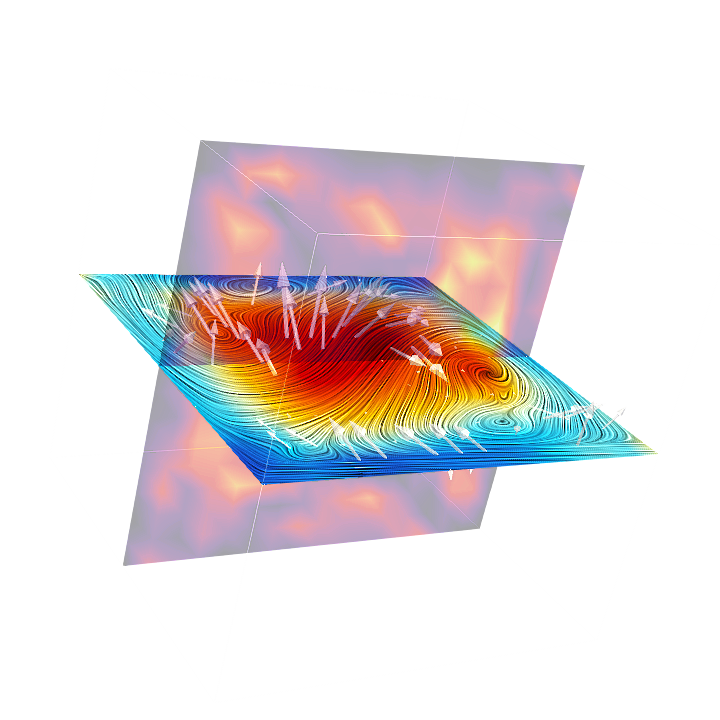}
		\caption{$t=0.05$}
	\end{subfigure}
	\begin{subfigure}[b]{0.07\textwidth}
		\centering
		\includegraphics[width=\textwidth]{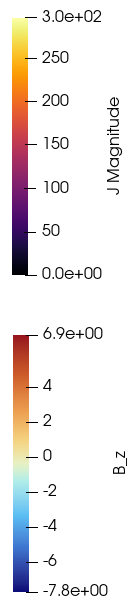}
	\end{subfigure}
	\begin{subfigure}[b]{0.45\textwidth}
		\centering
		\includegraphics[width=\textwidth]{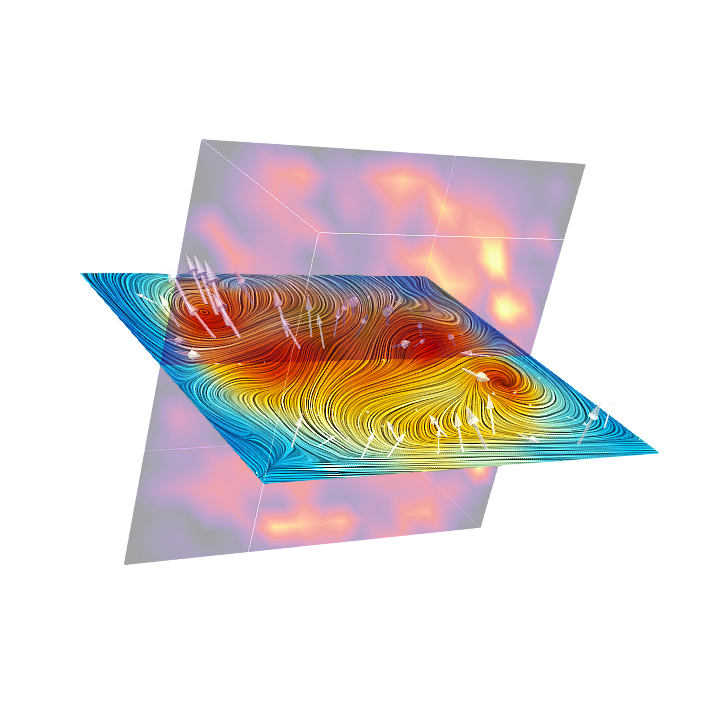}
		\caption{$t=0.1$}
	\end{subfigure}
	\begin{subfigure}[b]{0.45\textwidth}
		\centering
		\includegraphics[width=\textwidth]{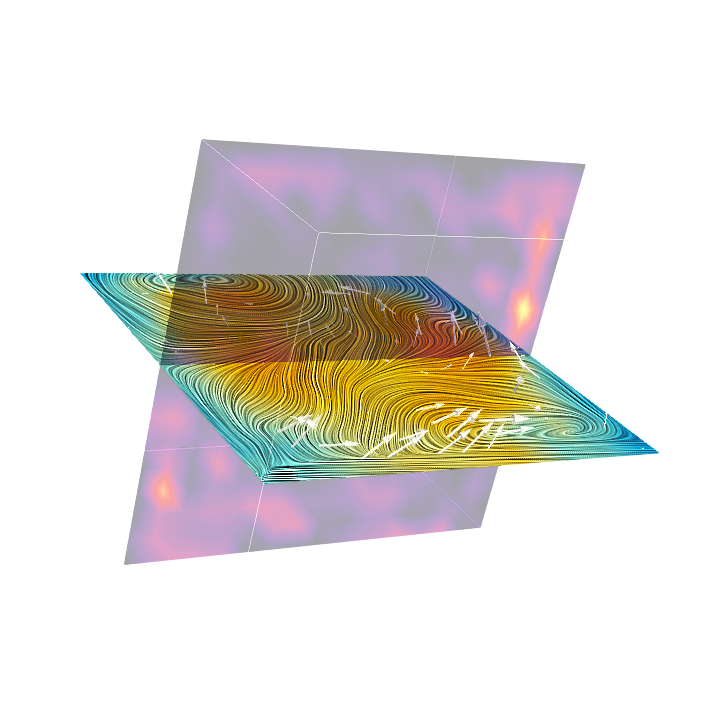}
		\caption{$t=0.2$}
	\end{subfigure}
	\begin{subfigure}[b]{0.07\textwidth}
		\centering
		\includegraphics[width=\textwidth]{exp1_0_legend.png}
	\end{subfigure}
	\caption{Snapshots of the 3D ABC flow simulation at given times. In the horizontal slice ($z=0.5$), background colouring indicates the out-of-plane magnetic field component $B_z$, while overlaid streamlines visualise the in-plane magnetic topology and vectors represent the fluid velocity $\bff{u}$. The vertical slice ($x=0.5$) displays the current density magnitude $|\mathbf{J}|$, identifying localised current ribbons.}
	\label{fig:snapshots exp 1}
\end{figure}

\begin{figure}[!htb]
	\centering
	\begin{subfigure}[b]{0.48\textwidth}
		\centering
		\includegraphics[width=\textwidth]{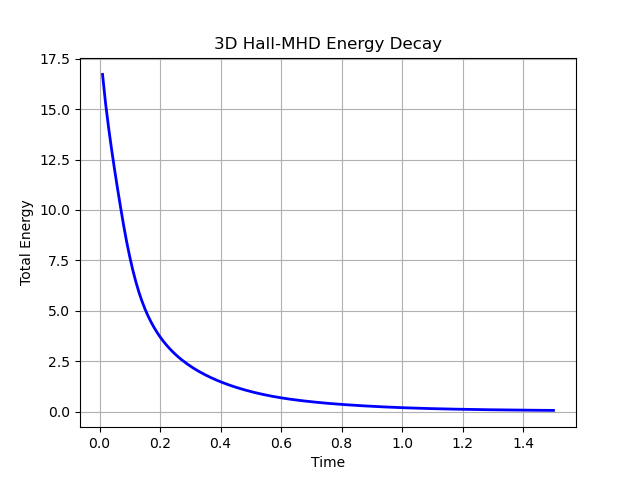}
	\end{subfigure}
	\hspace{1ex}
	\begin{subfigure}[b]{0.48\textwidth}
		\centering
		\includegraphics[width=\textwidth]{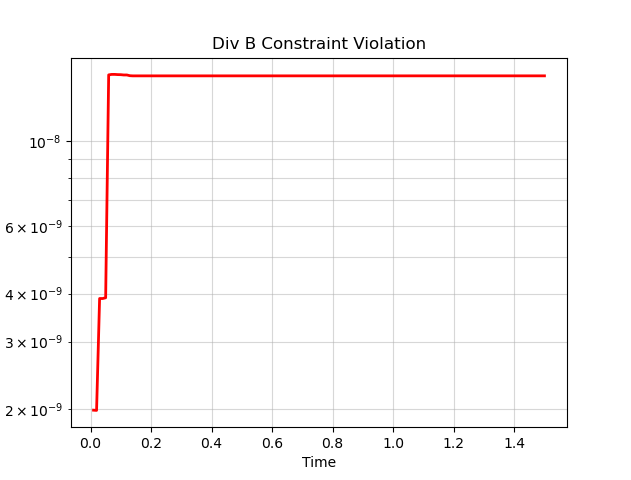}
	\end{subfigure}
	\caption{Energy and $\max_{\mathcal{T}_h} \abs{\divg \bff{B}}$ in Simulation~\ref{subsec:abc} with $h=1/16$ and $\tau=0.01$.}
	\label{fig:energy div exp1}
\end{figure}

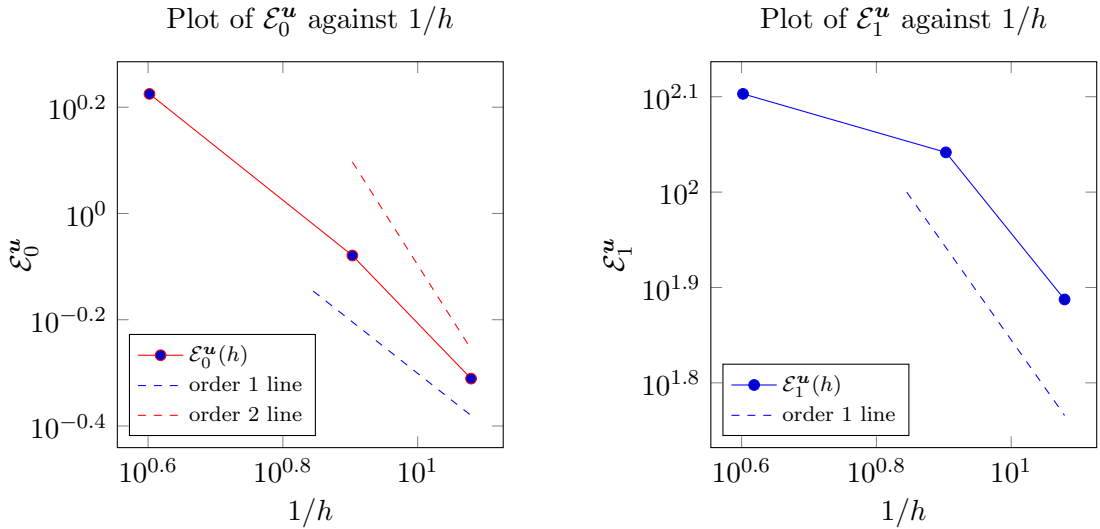
\begin{figure}[!htb]
	\begin{subfigure}[b]{0.45\textwidth}
		\centering
		\begin{tikzpicture}
			\begin{axis}[
				title=Plot of $\mathcal{E}_0^{\bff{u}}$ against $1/h$,
				height=0.9\textwidth,
				width=0.9\textwidth,
				xlabel= $1/h$,
				ylabel= $\mathcal{E}_0^{\bff{u}}$,
				xmode=log,
				ymode=log,
				legend pos=south west,
				legend cell align=left,
				]
				\addplot+[mark=*,red] coordinates {(4,1.679)(8,0.8335)(12,0.4889)};
				\addplot+[dashed,no marks,blue,domain=7:12]{5/x};
				\addplot+[dashed,no marks,red,domain=8:12]{80/x^2};
				\legend{\scriptsize{$\mathcal{E}_0^{\bff{u}}(h)$}, \scriptsize{order 1 line}, \scriptsize{order 2 line}}
			\end{axis}
		\end{tikzpicture}
	\end{subfigure}
	\hspace{1ex}
	\begin{subfigure}[b]{0.45\textwidth}
		\centering
		\begin{tikzpicture}
			\begin{axis}[
				title=Plot of $\mathcal{E}_1^{\bff{u}}$ against $1/h$,
				height=0.9\textwidth,
				width=0.9\textwidth,
				xlabel= $1/h$,
				ylabel= $\mathcal{E}_1^{\bff{u}}$,
				xmode=log,
				ymode=log,
				legend pos=south west,
				legend cell align=left,
				]
				\addplot+[mark=*,blue] coordinates {(4,126.8)(8,110.1)(12,77.2)};
				\addplot+[dashed,no marks,blue,domain=7:12]{700/x};
				\legend{\scriptsize{$\mathcal{E}_1^{\bff{u}}(h)$}, \scriptsize{order 1 line}}
			\end{axis}
		\end{tikzpicture}
	\end{subfigure}
	\caption{Spatial convergence orders of $\bff{u}$ in Simulation \ref{subsec:abc}.}
	\label{fig:conv u 0}
\end{figure}

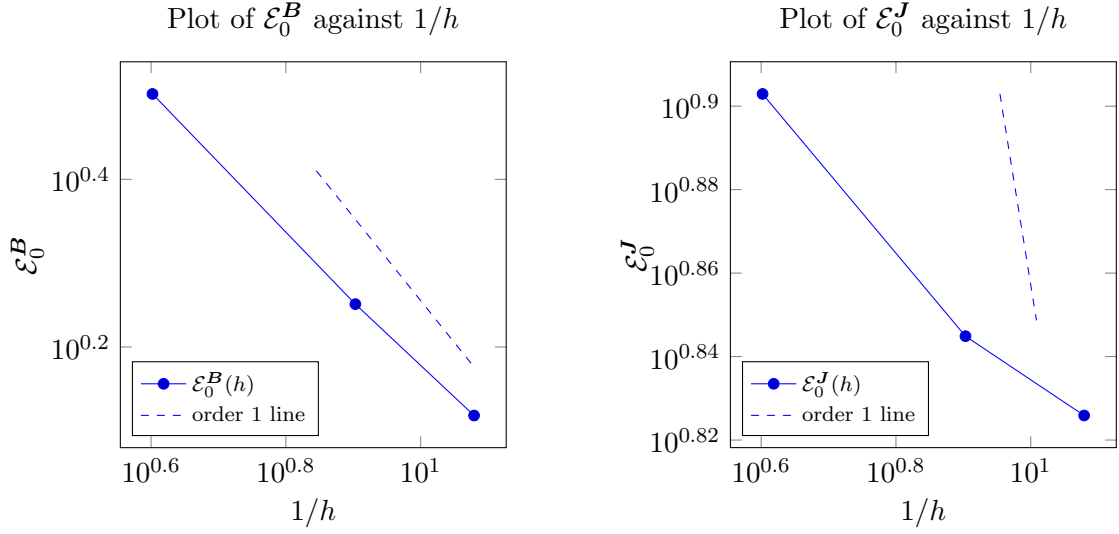
\begin{figure}[!htb]
	\begin{subfigure}[b]{0.45\textwidth}
		\centering
		\begin{tikzpicture}
			\begin{axis}[
				title=Plot of $\mathcal{E}_0^{\bff{B}}$ against $1/h$,
				height=0.9\textwidth,
				width=0.9\textwidth,
				xlabel= $1/h$,
				ylabel= $\mathcal{E}_0^{\bff{B}}$,
				xmode=log,
				ymode=log,
				legend pos=south west,
				legend cell align=left,
				]
				\addplot+[mark=*,blue] coordinates {(4,3.176)(8,1.783)(12,1.313)};
				\addplot+[dashed,no marks,blue,domain=7:12]{18/x};
				\legend{\scriptsize{$\mathcal{E}_0^{\bff{B}}(h)$}, \scriptsize{order 1 line}}
			\end{axis}
		\end{tikzpicture}
	\end{subfigure}
	\hspace{1em}
	\begin{subfigure}[b]{0.45\textwidth}
		\centering
		\begin{tikzpicture}
			\begin{axis}[
				title=Plot of $\mathcal{E}_0^{\bff{J}}$ against $1/h$,
				height=0.9\textwidth,
				width=0.9\textwidth,
				xlabel= $1/h$,
				ylabel= $\mathcal{E}_0^{\bff{J}}$,
				xmode=log,
				ymode=log,
				legend pos=south west,
				legend cell align=left,
				]
				\addplot+[mark=*,blue] coordinates {(4,8.0)(8,7.0)(12,6.7)};
				\addplot+[dashed,no marks,blue,domain=9:10.2]{72/x};
				\legend{\scriptsize{$\mathcal{E}_0^{\bff{J}}(h)$}, \scriptsize{order 1 line}}
			\end{axis}
		\end{tikzpicture}
	\end{subfigure}
	\caption{Spatial convergence orders of $\bff{B}$ and $\bff{J}$ in Simulation \ref{subsec:abc}.}
	\label{fig:conv b j 0}
\end{figure}

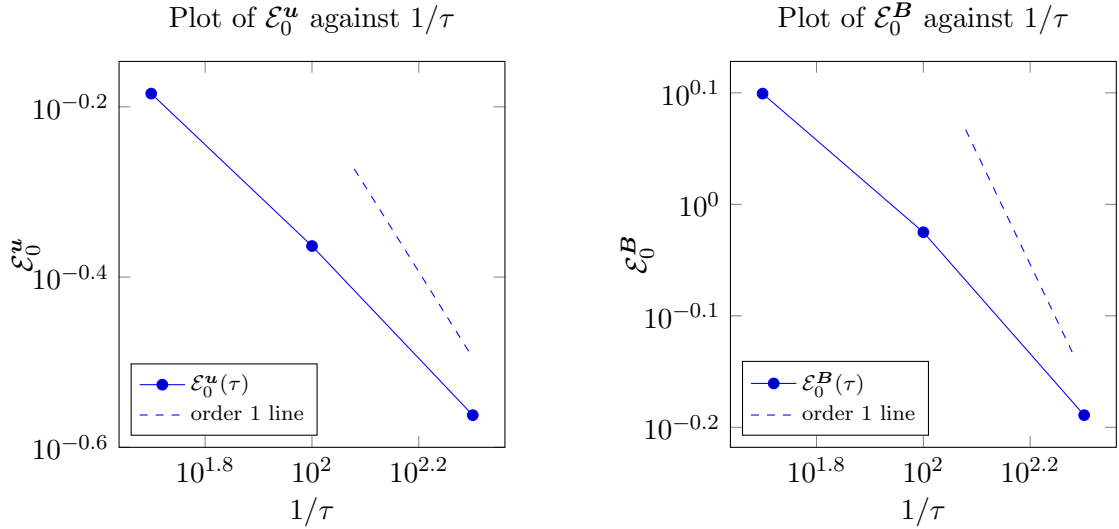
\begin{figure}[!htb]
	\begin{subfigure}[b]{0.45\textwidth}
		\centering
		\begin{tikzpicture}
			\begin{axis}[
				title=Plot of $\mathcal{E}_0^{\bff{u}}$ against $1/\tau$,
				height=0.9\textwidth,
				width=0.9\textwidth,
				xlabel= $1/\tau$,
				ylabel= $\mathcal{E}_0^{\bff{u}}$,
				xmode=log,
				ymode=log,
				legend pos=south west,
				legend cell align=left,
				]
				\addplot+[mark=*,blue] coordinates {(50,0.654)(100,0.433)(200,0.274)};
				\addplot+[dashed,no marks,blue,domain=120:200]{64/x};
				\legend{\scriptsize{$\mathcal{E}_0^{\bff{u}}(\tau)$}, \scriptsize{order 1 line}}
			\end{axis}
		\end{tikzpicture}
	\end{subfigure}
	\hspace{1em}
	\begin{subfigure}[b]{0.45\textwidth}
		\centering
		\begin{tikzpicture}
			\begin{axis}[
				title=Plot of $\mathcal{E}_0^{\bff{B}}$ against $1/\tau$,
				height=0.9\textwidth,
				width=0.9\textwidth,
				xlabel= $1/\tau$,
				ylabel= $\mathcal{E}_0^{\bff{B}}$,
				xmode=log,
				ymode=log,
				legend pos=south west,
				legend cell align=left,
				]
				\addplot+[mark=*,blue] coordinates {(50,1.257)(100,0.944)(200,0.647)};
				\addplot+[dashed,no marks,blue,domain=120:190]{140/x};
				\legend{\scriptsize{$\mathcal{E}_0^{\bff{B}}(\tau)$}, \scriptsize{order 1 line}}
			\end{axis}
		\end{tikzpicture}
	\end{subfigure}
	\caption{Temporal convergence orders of $\bff{u}$ and $\bff{B}$ in Simulation \ref{subsec:abc}.}
	\label{fig:temp conv u b 0}
\end{figure}

\subsection{The confined Orszag--Tang vortex}\label{subsec:tang}

We fix $\mathscr{D}=[0,1]\times [0,1]$ and consider the Orszag--Tang-type vortex problem for 2.5D Hall--MHD (in which all quantities are considered in 3D, but there is no dependence on the $z$-coordinate; see Section~\ref{sec:25d} for more details). The initial fluid velocity is set to be
\[
\bff{u}_0(x,y)= (-2.5\sin(2\pi y), 2.5\sin (2\pi x), 0).
\]
The initial magnetic field is defined as follows. Let 
\[
A_0(x,y,z)= \frac{1}{\pi} \sin (\pi x) \sin(\pi y) \left(\frac{1}{4} \cos (4\pi x)+ 2 \cos(2\pi y)\right)
\]
and define $\bff{B}_0=\curl (A_0 \hat{\bff{z}})$. We note that $\divg \bff{B}_0=0$ in $\mathscr{D}$ and $\bff{B}_0 \cdot \bff{n}=0$ on $\pa \mathscr{D}$.
The parameters of the problem are $\nu=\sigma=0.002$, $\eta=0.1$, $\alpha_1=10^{-8}$, and $\alpha_2=10^{-5}$.

Snapshots of $\bff{u}$ and $\bff{B}$ at selected times are shown in Figures~\ref{fig:snapshots u 2} and \ref{fig:snapshots b 2}, respectively. Snapshots of the $z$-component of $J$ are shown in Figure~\ref{fig:snapshots j 2}. Initially, the large-scale vortex drives the magnetic islands into a reconnection phase at around $t=0.14$. Due to the high Hall parameter $\eta$, we observe the formation of small-scale structures and secondary eddies early in time, consistent with Hall--MHD turbulence. Decreasing $\eta$ results in a delay of this phase. This is marked by a quadrupole pattern in the $z$-component of the current density $\bff{J}$ near the centre, signifying region of large gradients. As the simulation progresses, the system exhibits an inverse energy cascade.

Plots of energy vs time and $\max_{\mathcal{T}_h} \abs{\divg \bff{B}}$ vs time for $h=1/50$ and $\tau=0.005$ are shown in Figure~\ref{fig:energy div exp2}, confirming energy stability and divergence-preserving property (up to solver tolerance) of the scheme. Finally, plots of the errors of $\bff{u},\bff{B},\bff{J}$ against $1/h$ to verify spatial orders of convergence at $T=0.05$ are shown in Figures~\ref{fig:conv u 1} and~\ref{fig:conv b j 1}. The observed convergence rates for $\bff{u}$ and $\bff{J}$ are slightly reduced, which may be attributed to limited regularity of the solution. Plots of $\mathcal{E}_0^{\bff{u}}$ and $\mathcal{E}_0^{\bff{B}}$ against $1/\tau$ are displayed in Figure~\ref{fig:temp conv u b 1}.

\begin{figure}[!htb]
	\centering
	\begin{subfigure}[b]{0.22\textwidth}
		\centering
		\includegraphics[width=\textwidth]{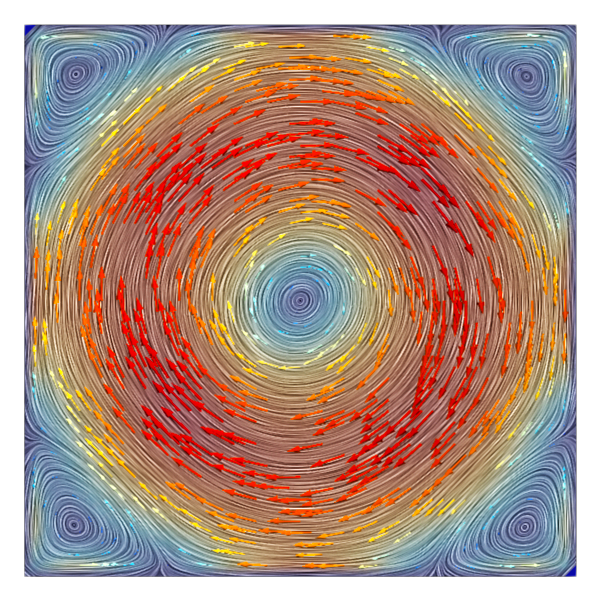}
		\caption{$t=0$}
	\end{subfigure}
	\begin{subfigure}[b]{0.22\textwidth}
		\centering
		\includegraphics[width=\textwidth]{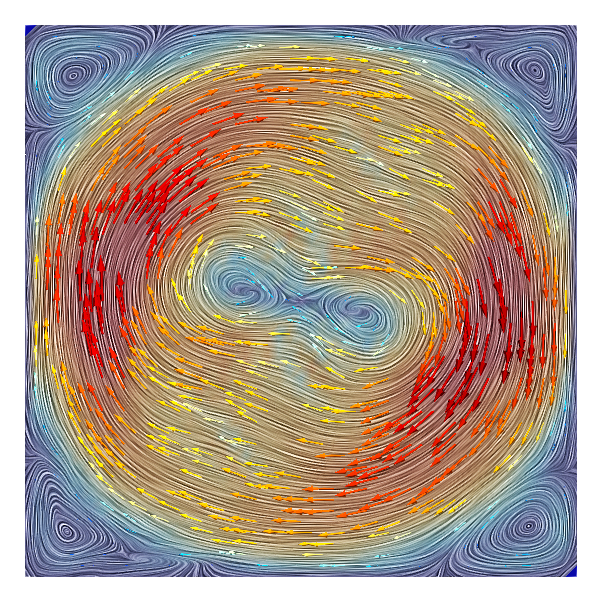}
		\caption{$t=0.1$}
	\end{subfigure}
	\begin{subfigure}[b]{0.22\textwidth}
		\centering
		\includegraphics[width=\textwidth]{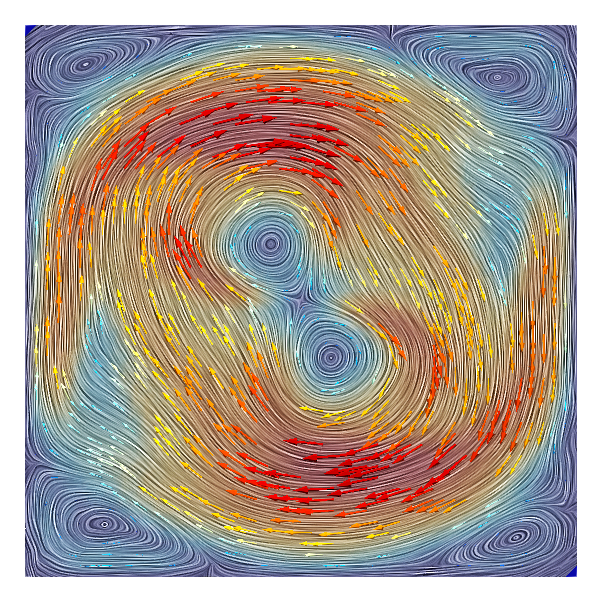}
		\caption{$t=0.2$}
	\end{subfigure}
	\begin{subfigure}[b]{0.22\textwidth}
		\centering
		\includegraphics[width=\textwidth]{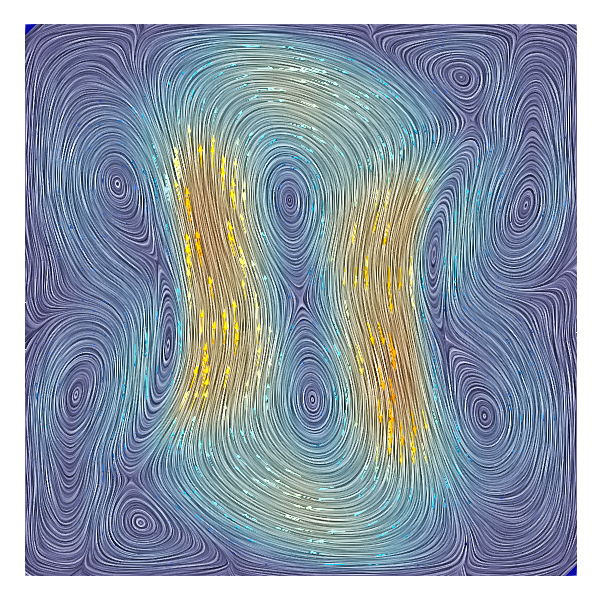}
		\caption{$t=1.0$}
	\end{subfigure}
	\begin{subfigure}[b]{0.08\textwidth}
		\centering
		\includegraphics[width=\textwidth]{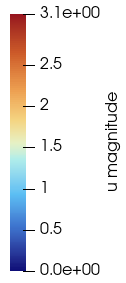}
	\end{subfigure}
	\caption{Streamlines of the fluid velocity $\bff{u}$ in Simulation~\ref{subsec:tang}. Background colour indicates $\abs{\bff{u}}$, and vectors represent $\bff{u}$ with lengths proportional to magnitude.}
	\label{fig:snapshots u 2}
\end{figure}

\begin{figure}[!htb]
	\centering
	\begin{subfigure}[b]{0.22\textwidth}
		\centering
		\includegraphics[width=\textwidth]{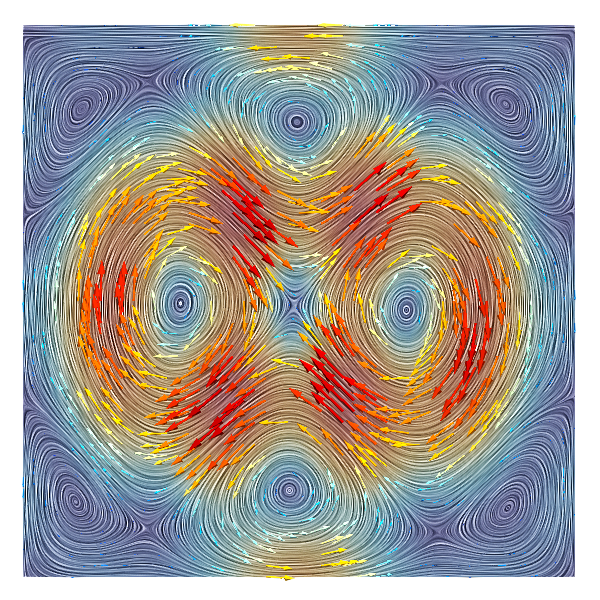}
		\caption{$t=0$}
	\end{subfigure}
	\begin{subfigure}[b]{0.22\textwidth}
		\centering
		\includegraphics[width=\textwidth]{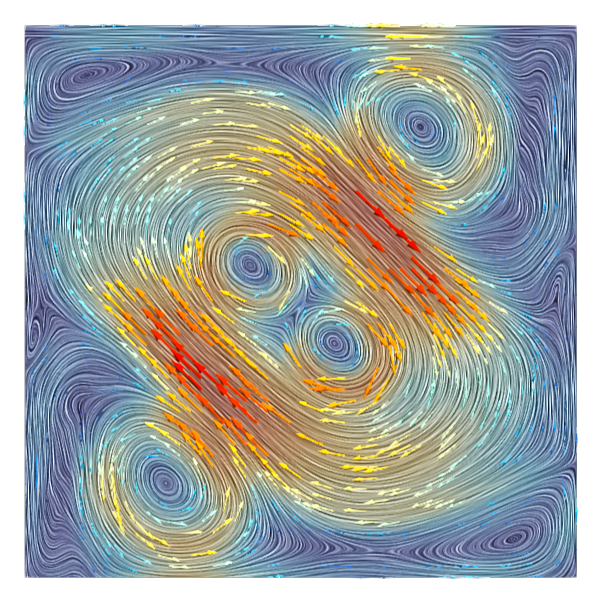}
		\caption{$t=0.1$}
	\end{subfigure}
	\begin{subfigure}[b]{0.22\textwidth}
		\centering
		\includegraphics[width=\textwidth]{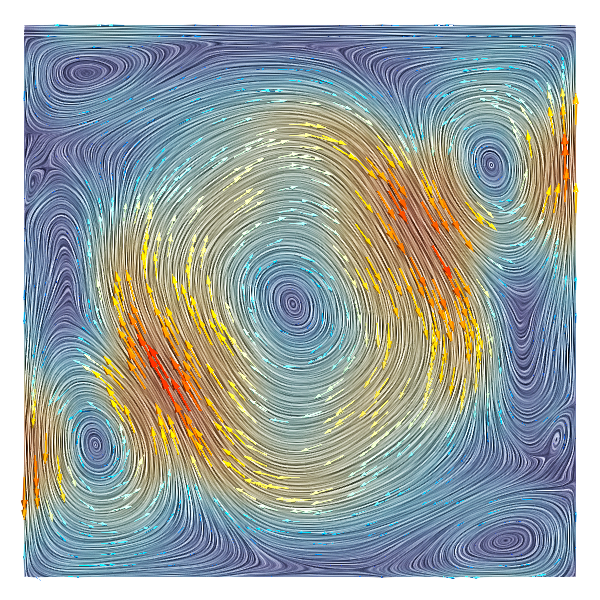}
		\caption{$t=0.2$}
	\end{subfigure}
	\begin{subfigure}[b]{0.22\textwidth}
		\centering
		\includegraphics[width=\textwidth]{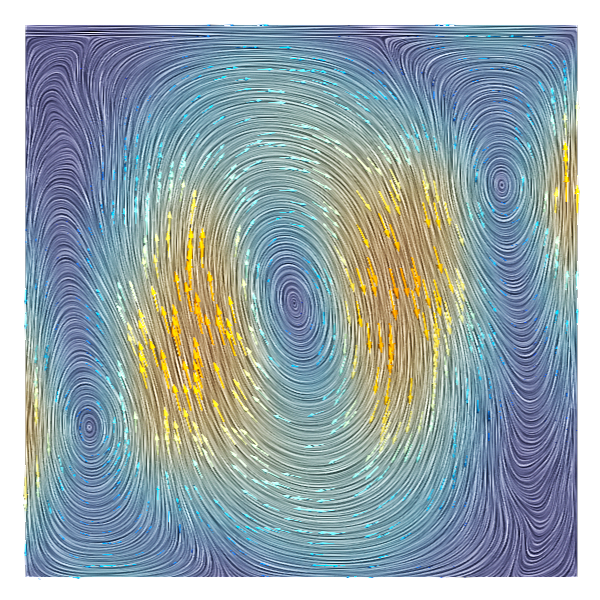}
		\caption{$t=1.0$}
	\end{subfigure}
	\begin{subfigure}[b]{0.08\textwidth}
		\centering
		\includegraphics[width=\textwidth]{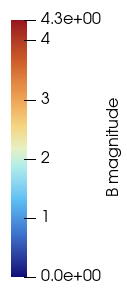}
	\end{subfigure}
	\caption{Streamlines of the magnetic field $\bff{B}$ in Simulation~\ref{subsec:tang}. Background colour indicates $\abs{\bff{B}}$, and vectors represent $\bff{B}$ with lengths proportional to magnitude.}
	\label{fig:snapshots b 2}
\end{figure}

\begin{figure}[!htb]
	\centering
	\begin{subfigure}[b]{0.22\textwidth}
		\centering
		\includegraphics[width=\textwidth]{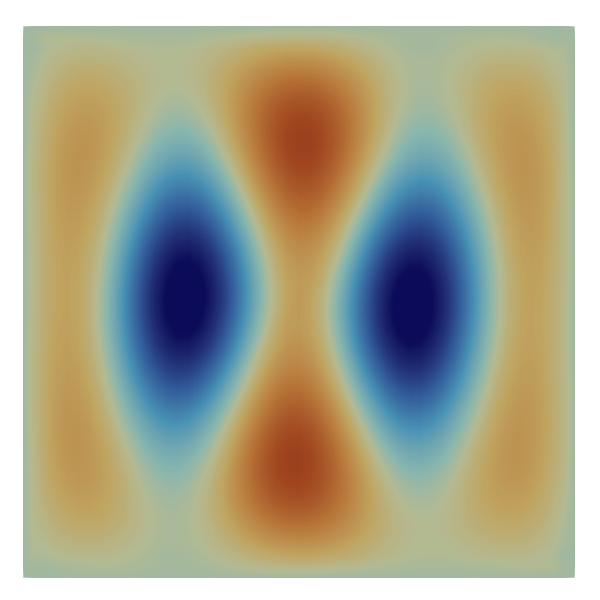}
		\caption{$t=0$}
	\end{subfigure}
	\begin{subfigure}[b]{0.22\textwidth}
		\centering
		\includegraphics[width=\textwidth]{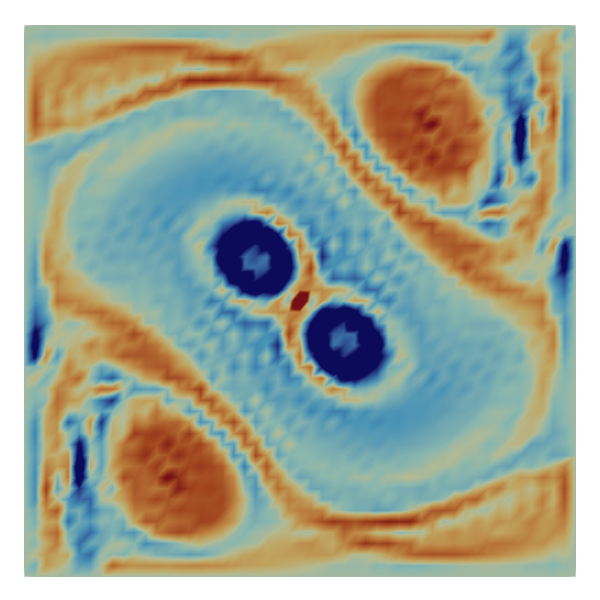}
		\caption{$t=0.1$}
	\end{subfigure}
	\begin{subfigure}[b]{0.22\textwidth}
		\centering
		\includegraphics[width=\textwidth]{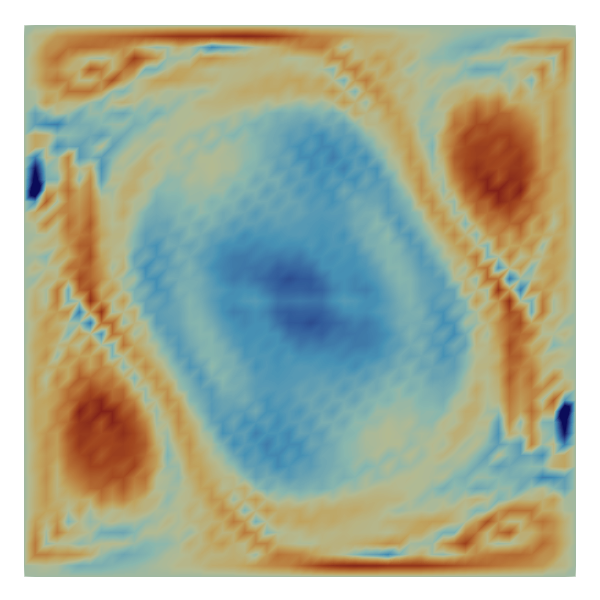}
		\caption{$t=0.2$}
	\end{subfigure}
	\begin{subfigure}[b]{0.22\textwidth}
		\centering
		\includegraphics[width=\textwidth]{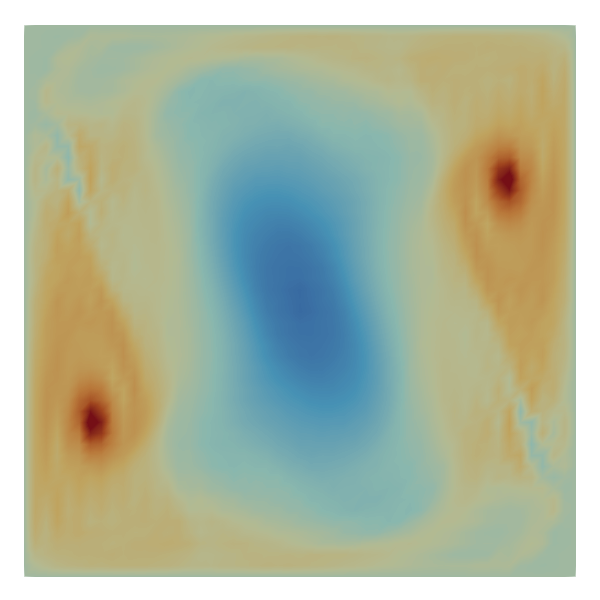}
		\caption{$t=1.0$}
	\end{subfigure}
	\begin{subfigure}[b]{0.08\textwidth}
		\centering
		\includegraphics[width=\textwidth]{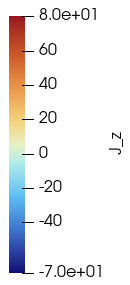}
	\end{subfigure}
	\caption{Snapshots of the $z$-component of $\bff{J}$ in Simulation~\ref{subsec:tang} at given times.}
	\label{fig:snapshots j 2}
\end{figure}

\begin{figure}[!htb]
	\centering
	\begin{subfigure}[b]{0.48\textwidth}
		\centering
		\includegraphics[width=\textwidth]{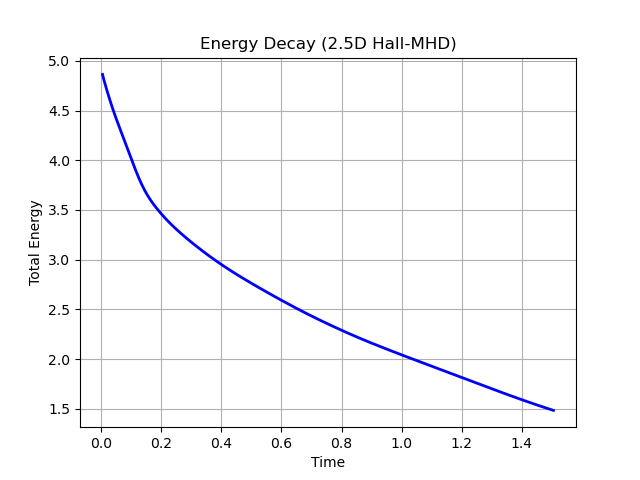}
	\end{subfigure}
	\hspace{1ex}
	\begin{subfigure}[b]{0.48\textwidth}
		\centering
		\includegraphics[width=\textwidth]{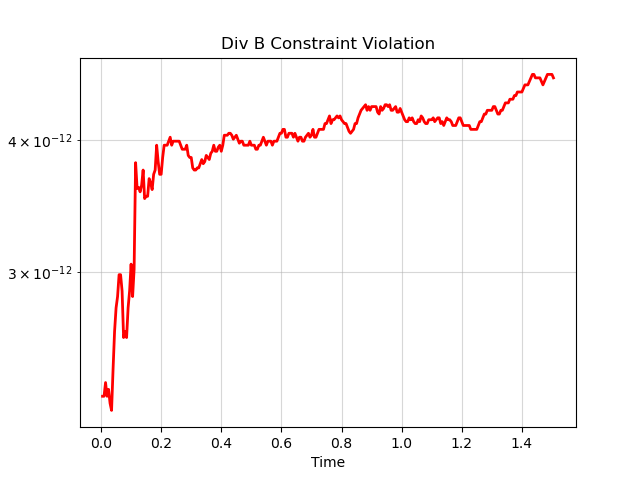}
	\end{subfigure}
	\caption{Energy and $\max_{\mathcal{T}_h} \abs{\divg \bff{B}}$ evolution in Simulation~\ref{subsec:tang} with $h=1/50$ and $\tau=0.005$.}
	\label{fig:energy div exp2}
\end{figure}

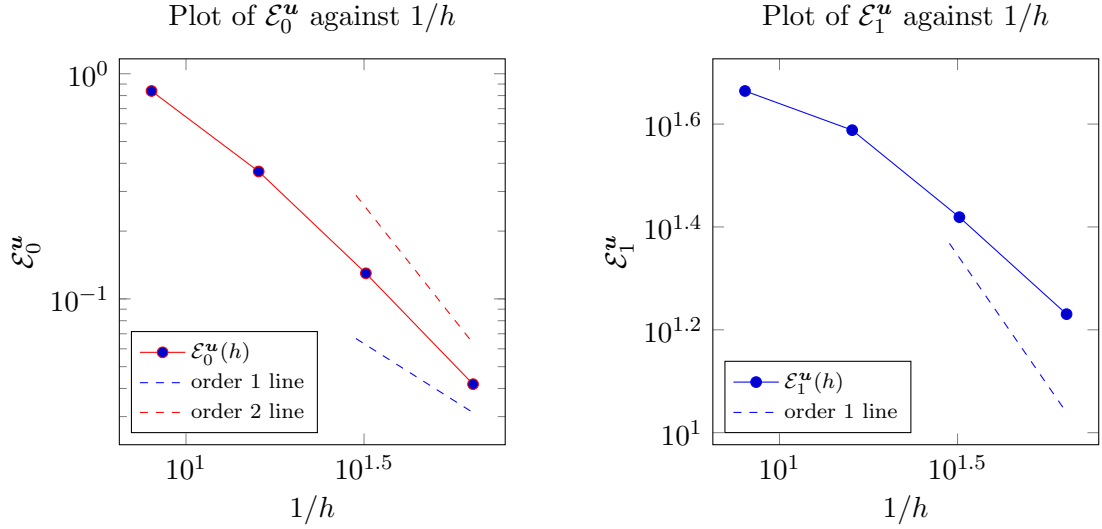
\begin{figure}[!htb]
	\begin{subfigure}[b]{0.45\textwidth}
		\centering
		\begin{tikzpicture}
			\begin{axis}[
				title=Plot of $\mathcal{E}_0^{\bff{u}}$ against $1/h$,
				height=0.9\textwidth,
				width=0.9\textwidth,
				xlabel= $1/h$,
				ylabel= $\mathcal{E}_0^{\bff{u}}$,
				xmode=log,
				ymode=log,
				legend pos=south west,
				legend cell align=left,
				]
				\addplot+[mark=*,red] coordinates {(8,0.8383)(16,0.3681)(32,0.1299)(64,0.0418)};
				\addplot+[dashed,no marks,blue,domain=30:64]{2/x};
				\addplot+[dashed,no marks,red,domain=30:64]{260/x^2};
				\legend{\scriptsize{$\mathcal{E}_0^{\bff{u}}(h)$}, \scriptsize{order 1 line}, \scriptsize{order 2 line}}
			\end{axis}
		\end{tikzpicture}
	\end{subfigure}
	\hspace{1ex}
	\begin{subfigure}[b]{0.45\textwidth}
		\centering
		\begin{tikzpicture}
			\begin{axis}[
				title=Plot of $\mathcal{E}_1^{\bff{u}}$ against $1/h$,
				height=0.9\textwidth,
				width=0.9\textwidth,
				xlabel= $1/h$,
				ylabel= $\mathcal{E}_1^{\bff{u}}$,
				xmode=log,
				ymode=log,
				legend pos=south west,
				legend cell align=left,
				]
				\addplot+[mark=*,blue] coordinates {(8,46.16)(16,38.74)(32,26.24)(64,17.0)};
				\addplot+[dashed,no marks,blue,domain=30:64]{700/x};
				\legend{\scriptsize{$\mathcal{E}_1^{\bff{u}}(h)$}, \scriptsize{order 1 line}}
			\end{axis}
		\end{tikzpicture}
	\end{subfigure}
	\caption{Spatial convergence orders of $\bff{u}$ in Simulation \ref{subsec:tang}.}
	\label{fig:conv u 1}
\end{figure}

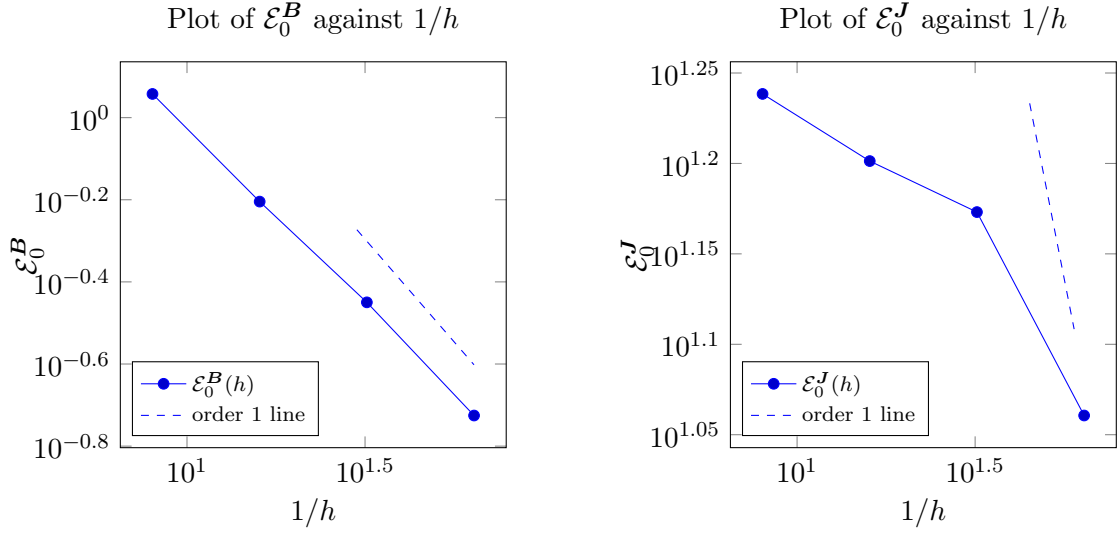
\begin{figure}[!htb]
	\begin{subfigure}[b]{0.45\textwidth}
		\centering
		\begin{tikzpicture}
			\begin{axis}[
				title=Plot of $\mathcal{E}_0^{\bff{B}}$ against $1/h$,
				height=0.9\textwidth,
				width=0.9\textwidth,
				xlabel= $1/h$,
				ylabel= $\mathcal{E}_0^{\bff{B}}$,
				xmode=log,
				ymode=log,
				legend pos=south west,
				legend cell align=left,
				]
				\addplot+[mark=*,blue] coordinates {(8,1.1421)(16,0.6245)(32,0.3550)(64,0.1882)};
				\addplot+[dashed,no marks,blue,domain=30:64]{16/x};
				\legend{\scriptsize{$\mathcal{E}_0^{\bff{B}}(h)$}, \scriptsize{order 1 line}}
			\end{axis}
		\end{tikzpicture}
	\end{subfigure}
	\hspace{1em}
	\begin{subfigure}[b]{0.45\textwidth}
		\centering
		\begin{tikzpicture}
			\begin{axis}[
				title=Plot of $\mathcal{E}_0^{\bff{J}}$ against $1/h$,
				height=0.9\textwidth,
				width=0.9\textwidth,
				xlabel= $1/h$,
				ylabel= $\mathcal{E}_0^{\bff{J}}$,
				xmode=log,
				ymode=log,
				legend pos=south west,
				legend cell align=left,
				]
				\addplot+[mark=*,blue] coordinates {(8,17.32)(16,15.9)(32,14.9)(64,11.5)};
				\addplot+[dashed,no marks,blue,domain=45:60]{770/x};
				\legend{\scriptsize{$\mathcal{E}_0^{\bff{J}}(h)$}, \scriptsize{order 1 line}}
			\end{axis}
		\end{tikzpicture}
	\end{subfigure}
	\caption{Spatial convergence orders of $\bff{B}$ and $\bff{J}$ in Simulation \ref{subsec:tang}.}
	\label{fig:conv b j 1}
\end{figure}

\begin{figure}[!htb]
	\begin{subfigure}[b]{0.45\textwidth}
		\centering
		\begin{tikzpicture}
			\begin{axis}[
				title=Plot of $\mathcal{E}_0^{\bff{u}}$ against $1/\tau$,
				height=0.9\textwidth,
				width=0.9\textwidth,
				xlabel= $1/\tau$,
				ylabel= $\mathcal{E}_0^{\bff{u}}$,
				xmode=log,
				ymode=log,
				legend pos=south west,
				legend cell align=left,
				]
				\addplot+[mark=*,blue] coordinates {(40,0.5257)(80,0.3467)(160,0.2398)(320,0.1007)};
				\addplot+[dashed,no marks,blue,domain=150:320]{50/x};
				\legend{\scriptsize{$\mathcal{E}_0^{\bff{u}}(\tau)$}, \scriptsize{order 1 line}}
			\end{axis}
		\end{tikzpicture}
	\end{subfigure}
	\hspace{1em}
	\begin{subfigure}[b]{0.45\textwidth}
		\centering
		\begin{tikzpicture}
			\begin{axis}[
				title=Plot of $\mathcal{E}_0^{\bff{B}}$ against $1/\tau$,
				height=0.9\textwidth,
				width=0.9\textwidth,
				xlabel= $1/\tau$,
				ylabel= $\mathcal{E}_0^{\bff{B}}$,
				xmode=log,
				ymode=log,
				legend pos=south west,
				legend cell align=left,
				]
				\addplot+[mark=*,blue] coordinates {(40,1.045)(80,0.7188)(160,0.339)(320,0.222)};
				\addplot+[dashed,no marks,blue,domain=150:320]{85/x};
				\legend{\scriptsize{$\mathcal{E}_0^{\bff{B}}(\tau)$}, \scriptsize{order 1 line}}
			\end{axis}
		\end{tikzpicture}
	\end{subfigure}
	\caption{Temporal convergence orders of $\bff{u}$ and $\bff{B}$ in Simulation \ref{subsec:tang}.}
	\label{fig:temp conv u b 1}
\end{figure}
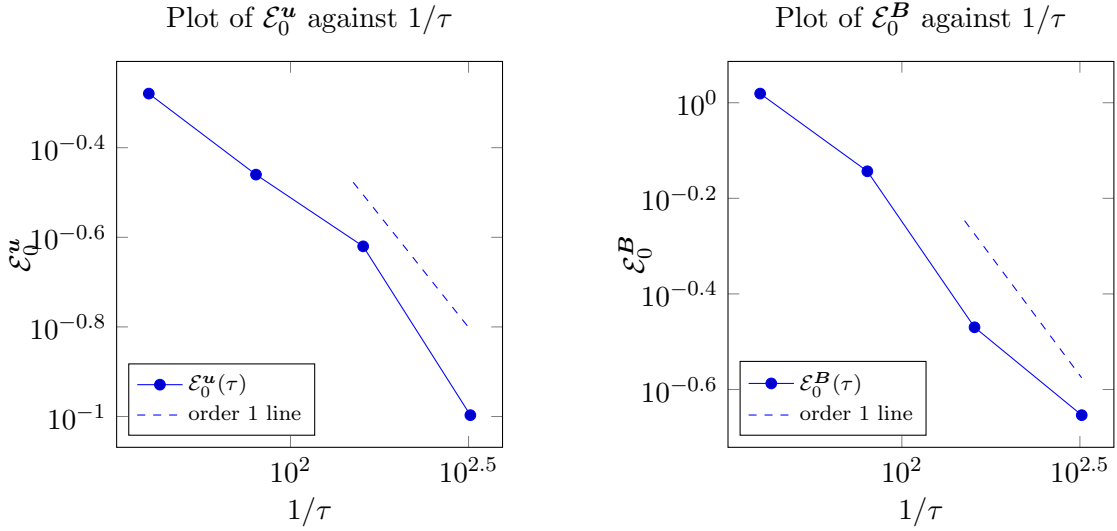

\subsection{Hall-driven reconnection in a confined modulated Harris sheet}\label{subsec:hall harris}

We consider a confined Harris-type current sheet in a three-dimensional domain $\mathscr{D}=[0,1]^3$, designed to study magnetic reconnection in Hall--MHD under perfectly conducting boundary conditions. The initial fluid velocity is taken to be zero so that the dynamics are driven purely by magnetic forces. The initial magnetic field is constructed from a scalar magnetic potential $A_0$ of the form
\[
A_0(x,y,z)=B_0 \delta \sin(\pi x) \sin(\pi y) \sin(\pi z) \log \cosh\left(\frac{y-0.5}{\delta}\right),
\] 
and we define $\bff{B}_0=\curl (A_0 \hat{\bff{z}})$. By construction, the magnetic field is divergence-free in $\mathscr{D}$ and satisfies $\bff{B}_0\cdot \bff{n}=0$ on $\pa \mathscr{D}$. The parameter $\delta>0$ controls the thickness of the current sheet, while the sinusoidal modulation confines the sheet to the interior of $\mathscr{D}$ and avoids boundary artefacts. The parameters are chosen to balance numerical stability with the ability to resolve Hall-driven reconnection on coarse meshes. In particular, we set $\nu=0.004$, $\sigma=0.008$, $\eta=0.15$, and $\alpha_1=\alpha_2=10^{-5}$. The simulation is performed on a uniform mesh with $h=1/16$ and $\tau=0.01$.

Snapshots of the magnetic field and fluid velocity at representative times are shown in Figure~\ref{fig:snapshots 3 Hall} for $\eta\neq 0$ and in Figure~\ref{fig:snapshots 3 no Hall} for $\eta=0$. As the simulation evolves, the initially planar current sheet thins near the centre of the domain and undergoes magnetic reconnection, leading to the formation of magnetic islands in the plane $z=0.5$. In contrast to resistive (non-Hall) MHD, the Hall term accelerates the reconnection process and generates localised small-scale structures around the reconnection region. In particular, a characteristic quadrupolar pattern in selected components of $\bff{J}$ is observed near the reconnection site, which is a hallmark of Hall--MHD dynamics.

The evolution of the energy and the maximum divergence error $\max_{\mathcal{T}_h} \abs{\divg \bff{B}}$ is shown in Figure~\ref{fig:energy div exp3} . The results confirm that the scheme remains energy-stable and preserves the divergence-free constraint up to solver tolerance, even with relatively coarse spatial and temporal discretisation. This experiment therefore demonstrates the robustness of the proposed scheme in a fully three-dimensional reconnection setting driven by Hall effects.

\begin{figure}[!htb]
	\centering
	\begin{subfigure}[b]{0.45\textwidth}
		\centering
		\includegraphics[width=\textwidth]{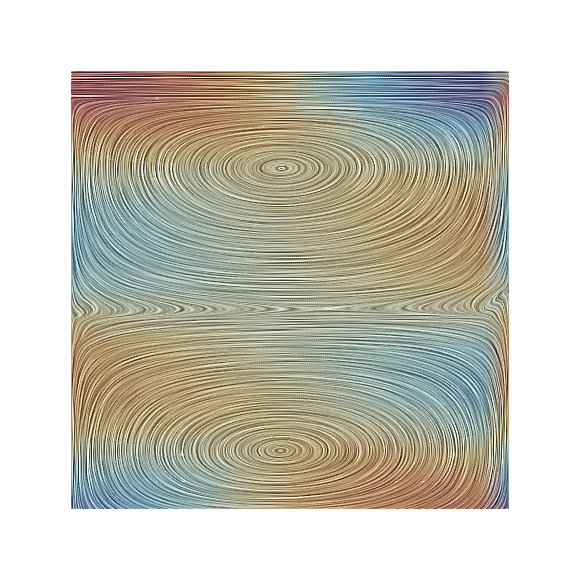}
		\caption{$t=0$}
	\end{subfigure}
	\begin{subfigure}[b]{0.45\textwidth}
		\centering
		\includegraphics[width=\textwidth]{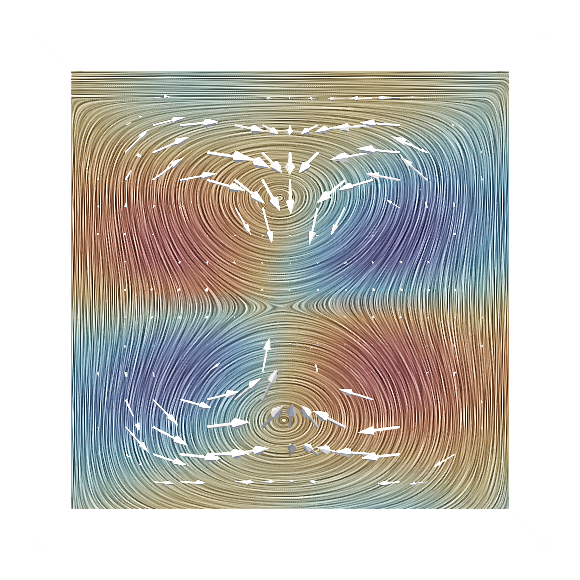}
		\caption{$t=0.1$}
	\end{subfigure}
	\begin{subfigure}[b]{0.07\textwidth}
		\centering
		\includegraphics[width=\textwidth]{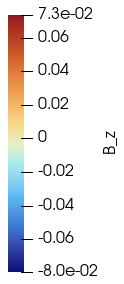}
	\end{subfigure}
	\begin{subfigure}[b]{0.45\textwidth}
		\centering
		\includegraphics[width=\textwidth]{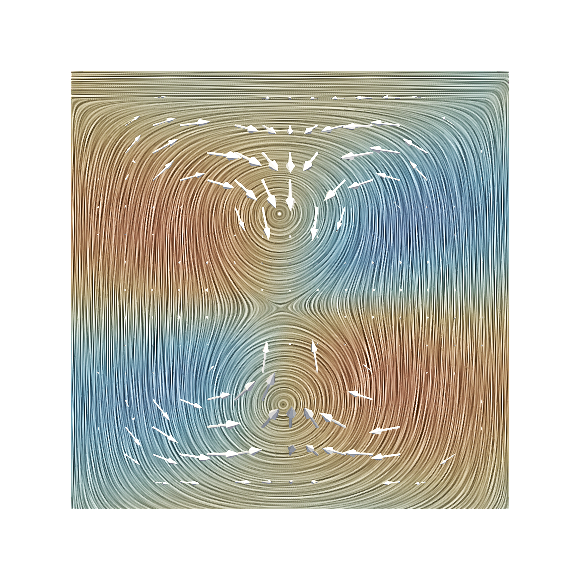}
		\caption{$t=0.15$}
	\end{subfigure}
	\begin{subfigure}[b]{0.45\textwidth}
		\centering
		\includegraphics[width=\textwidth]{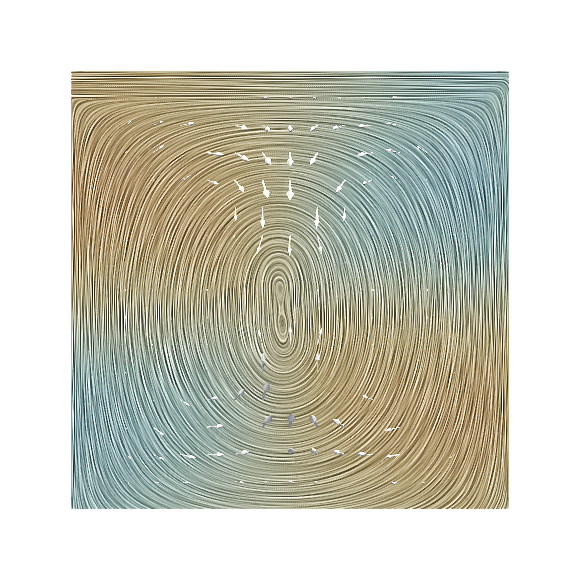}
		\caption{$t=0.25$}
	\end{subfigure}
	\begin{subfigure}[b]{0.07\textwidth}
		\centering
		\includegraphics[width=\textwidth]{hall_exp3_legend.png}
	\end{subfigure}
	\caption{Snapshots of the 3D confined Harris-sheet simulation on the plane $z=0.5$ at selected times with nonzero Hall parameter. Streamlines depict the in-plane magnetic field topology, revealing magnetic islands and X-points. The background colour shows the out-of-plane magnetic field component $B_z$, while arrows indicate the fluid velocity $\bff{u}$. A clear quadrupolar $B_z$ structure develops during island merging, characteristic of Hall-mediated reconnection.}
	\label{fig:snapshots 3 Hall}
\end{figure}

\begin{figure}[!htb]
	\centering
	\begin{subfigure}[b]{0.45\textwidth}
		\centering
		\includegraphics[width=\textwidth]{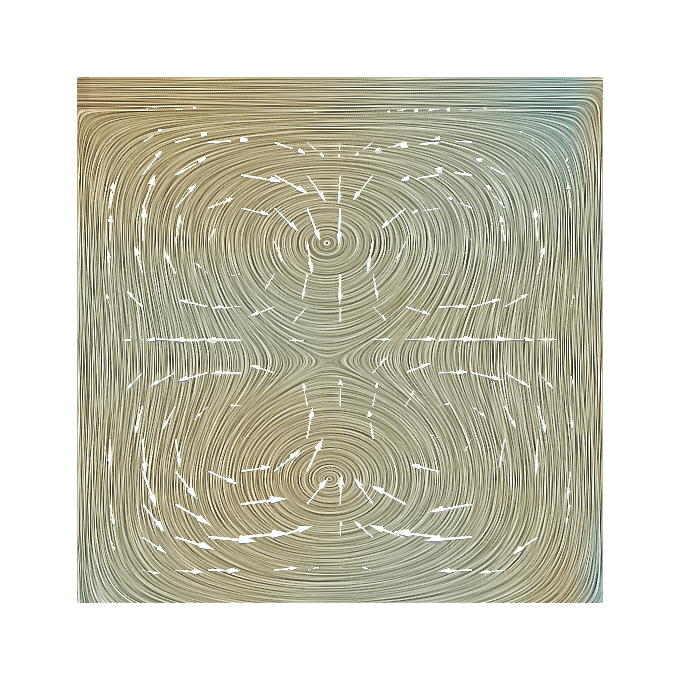}
		\caption{$t=0.15$}
	\end{subfigure}
	\begin{subfigure}[b]{0.45\textwidth}
		\centering
		\includegraphics[width=\textwidth]{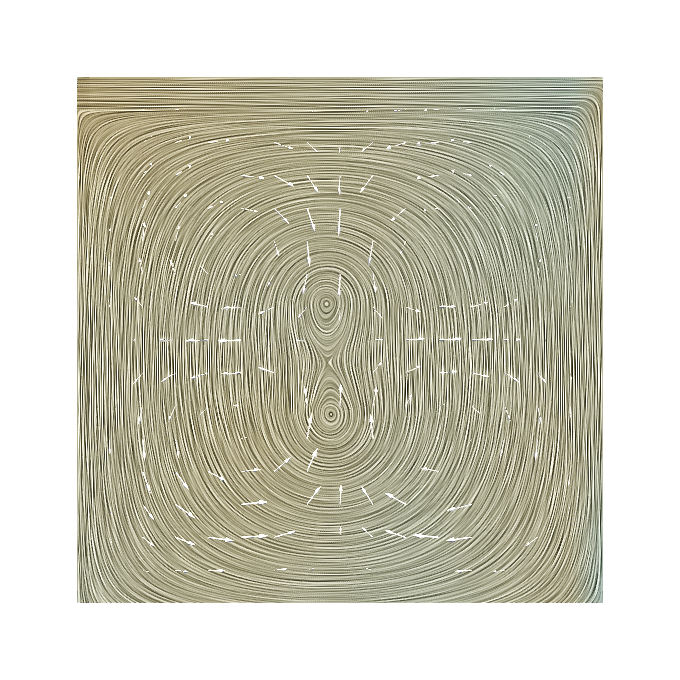}
		\caption{$t=0.25$}
	\end{subfigure}
	\begin{subfigure}[b]{0.07\textwidth}
		\centering
		\includegraphics[width=\textwidth]{hall_exp3_legend.png}
	\end{subfigure}
	\caption{Snapshots of the 3D confined Harris-sheet simulation on the plane~$z=0.5$ at selected times without Hall effects ($\eta=0$). Streamlines depict the in-plane magnetic field topology, showing magnetic islands and X-points. The background colour represents the out-of-plane magnetic field component $B_z$, while arrows indicate the fluid velocity $\bff{u}$. In contrast to the Hall--MHD case, no quadrupolar $B_z$ structure is observed, and magnetic island merging proceeds more slowly.}
	\label{fig:snapshots 3 no Hall}
\end{figure}

\begin{figure}[!htb]
	\centering
	\begin{subfigure}[b]{0.48\textwidth}
		\centering
		\includegraphics[width=\textwidth]{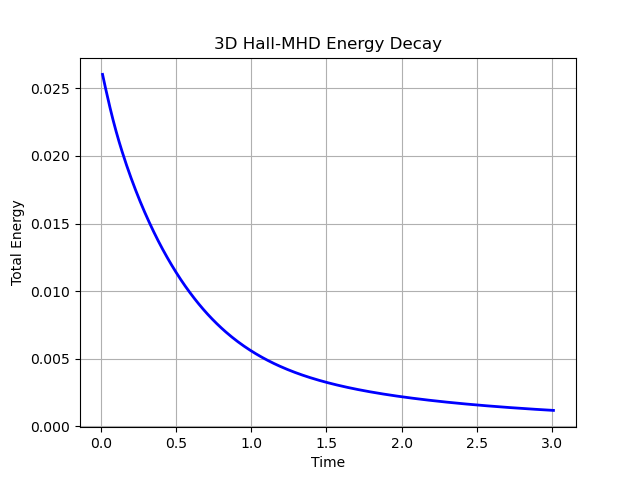}
	\end{subfigure}
	\hspace{1ex}
	\begin{subfigure}[b]{0.48\textwidth}
		\centering
		\includegraphics[width=\textwidth]{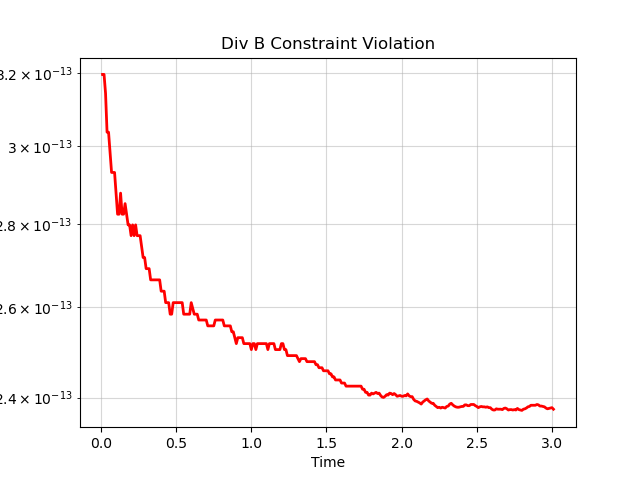}
	\end{subfigure}
	\caption{Energy and $\max_{\mathcal{T}_h} \abs{\divg \bff{B}}$ in Simulation~\ref{subsec:hall harris} with $h=1/16$ and $\tau=0.01$.}
	\label{fig:energy div exp3}
\end{figure}


\end{document}